\documentclass[10pt]{amsart}

\usepackage{graphicx,amssymb}
\usepackage{amsthm}
\usepackage{dsfont}
\usepackage{stmaryrd}
\usepackage{pstricks}
\usepackage{color}

\theoremstyle{definition}

\newtheorem{theorem}{Theorem}[section]
\newtheorem{definition}[theorem]{Definition}
\newtheorem{lemma}[theorem]{Lemma}
\newtheorem{corollary}[theorem]{Corollary}
\newtheorem{proposition}[theorem]{Proposition}
\newtheorem{claim}[theorem]{Claim}
\newtheorem{example}[theorem]{Example}
\newtheorem{conjecture}[theorem]{Conjecture}

\newtheorem{remark}[theorem]{Remark}

\newtheorem{acknowledgements}[theorem]{Acknowledgments}

\numberwithin{equation}{section}

\DeclareMathOperator{\Con}{Con}
\DeclareMathOperator{\Ab}{Ab}
\DeclareMathOperator{\End}{End}
\DeclareMathOperator{\Hom}{Hom}

\DeclareMathOperator{\Cg}{Cg}

\DeclareMathOperator{\Pol}{Pol}
\DeclareMathOperator{\id}{id}

\begin{document}

\title{On Supernilpotent Algebras}
\author{Alexander Wires}
\address{School of Economic Mathematics, Southwestern University of Finance and Economics\\ 555 Liutai Avenue, Wenjiang District\\
Chengdu 611130, Sichuan, China}
\email{awires@swufe.edu.cn}
\date{September 25, 2017}

\begin{abstract}
We establish a characterization of supernilpotent Mal'cev algebras which generalizes the affine structure of abelian Mal'cev algebras and the recent characterization of 3-supernilpotent Mal'cev algebras. We then show that for varieties in which the two-generated free algebra is finite: (1) neutrality of the higher commutators is equivalent to congruence meet-semidistributivity, and (2) the class of varieties which interpret a Mal'cev term in every supernilpotent algebra is equivalent to the existence of a weak difference term. We then establish properties of the higher commutator in the aforementioned second class of varieties.  
\end{abstract}

\maketitle


\section{Introduction}
\vspace{0.5cm}

The notion of centralizing elements and centralizing subgroups plays a fundamental role in the structure theory of groups, both finite and infinite. One of the most successful approaches in extending aspects of these ideas to more general algebraic systems has been the generalization of the commutator as a binary operation on congruence lattices of algebras.  The commutator theory for congruence permutable varieties developed by J.D.H Smith \cite{smith}, and latter extended to congruence modular varieties in Herrmann \cite{herrmann}, Hagemann and Herrmann \cite{hhcomm} and Gumm \cite{gumm}, satisfies many of the same useful properties as the traditional commutator when specialized to normal subgroups in the theory of groups. In the setting of congruence modular varieties, the commutator has become a powerful tool in understanding the structure of algebras in these varieties.   

The commutator, also known as the term-condition commutator, can be defined for arbitrary universal algebras, and while it is no longer as well-behaved as in the congruence modular setting, it has proven instrumental in numerous advances; for just two examples, in the classification of finite algebras by tame congruence theory \cite{TCT} and then in the extensions of many of these results to an infinite setting in Kearnes and Szendrei \cite{twocommutator} and Kearnes and Kiss \cite{shape}. This introduction can in no way give an accurate account and due recognition to all the contributors to the general commutator theory - please consult Freese and McKenzie \cite{commod}. 

Underlying many of the most impressive applications of the commutator for general algebras is the fact that it is defined by a certain centralizer relation universally quantified over the polynomials of an algebra which in principle carries information about the identities and behavior of the term operations in relation to particular elements. A fundamental example of this is illustrated by the result in \cite{herrmann} which establishes the affine structure of polynomials of an abelian algebra in a congruence modular variety; namely, an algebra in a congruence modular variety is abelian if and only if it is polynomially equivalent to an abelian group with a unital ring of operators.

The higher commutators defined by Bulatov \cite{bulatov} were introduced as a tool to help differentiate between distinct Mal'cev clones on a finite set. Akin to the binary term-condition commutator, they are defined for arbitrary algebras by a centralizer relation which in general satisfies many but not all of the properties of the binary instance; however, in Aichinger and Mudrinski \cite{supernil} we find that in congruence permutable varieties many of the properties that made the binary commutator such a useful tool continue to hold for the higher commutators. The 3-supernilpotent Mal'cev algebras (abelian in the sense of the ternary commutator) are characterized in Mudrinski \cite{mudrinski} as polynomially equivalent to groups with a special family of unary and binary operators. This is a proper generalization of the affine structure of abelian Mal'cev algebras from \cite{herrmann}. The role of supernilpotence in preventing the dualizability of finite Mal'cev algebras is investigated in Dentz and Mayr \cite{dentzmayr}, and in Mayr \cite{mayr} supernilpotence is an essential tool in establishing implicit descriptions of the polynomials of certain classes of finite local rings and groups.

In Opr\v{s}al \cite{oprsal}, the higher commutators for Mal'cev algebras are encoded by a certain subpower with applications to supernilpotent clones. Recently, in Moorhead \cite{moorhead} some of the fundamental properties of the higher commutator for congruence permutable varieties are extended to the congruence modular setting.

The main purpose of the present article is to provide a characterization of supernilpotent Mal'cev algebras which generalizes the affine structure of abelian Mal'cev algebras and the special class of expanded groups described in \cite{mudrinski}. The argument utilizes two principle ideas: the use of absorbing polynomials to mimic aspects of the commutator calculus in groups and the representation of arbitrary polynomials in loops by absorbing polynomials. In \cite{supernil}, we learn the utility of absorbing polynomials in proving many of the higher commutator properties in Mal'cev algebras; in particular, absorbing polynomials provide a convenient generating set of the higher commutators of congruences \cite[Lem.6.9]{supernil}. In \cite{mudrinski}, we see how absorbing polynomials are used to represent the polynomials of certain expanded groups and determine ``distributivity'' of smaller arity absorbing polynomials.

For an arbitrary variety, congruence meet-semidistributivity is equivalent to the neutrality of the binary term-condition commutator \cite{twocommutator}. We conjecture that congruence meet-semidistributivity is equivalent to the condition that all the higher commutators are neutral. This is established for congruence distributive varieties (Theorem \ref{thm:cdneutral}) and for arbitrary varieties under the assumption that the two-generated free algebra is finite (Theorem \ref{thm:neutral}). Under this same finiteness assumption, we also show the the class of varieties in which the supernilpotent algebras are polynomially equivalent to n-type loops (defined in Section \ref{sec:charct}) is equivalent to the class of varieties with a weak difference term (Theorem \ref{thm:ndifference}). We again conjecture that the finiteness assumption is unnecessary, and as evidence for the general conjecture, we establish the difficult part of the characterization for the class of congruence modular varieties (Theorem \ref{thm:cmdiff}).

In section \ref{sec:prelim}, we state the definitions and properties of the higher centralizers and commutators. We also record the characterizations of 3-supernilpotent Mal'cev algebras from \cite{mudrinski} and a central fact about nilpotent algebras in Mal'cev varieties which we will utilize. In section \ref{sec:charct}, we state and prove Theorem \ref{thm:nabelian} which is the characterization of arbitrary supernilpotent Mal'cev algebras. We also show how the characterizations for abelian and 3-supernilpotent Mal'cev algebras can be directly derived from Theorem \ref{thm:nabelian}. Section \ref{sec:neutral} contains the results related to congruence meet-semidistributivity and varieties with a weak difference term. In the final Section \ref{sec:weakdiffcomm}, we look toward future applications in varieties with weak m-difference terms by extending to the more general setting restricted instances of commutator properties from Mal'cev varieties (Theorem \ref{thm:ftermprop} and Theorem \ref{thm:nweak}).

\section{All and Sundry}\label{sec:prelim}

We record some facts which we will require in the next section. The definitions and notations for algebras are standard and can be found in Bergman \cite{bergman} and Burris and Sankappanavar \cite{univ}. For definitions and notation concerning the commutator consult \cite{commod}.

The natural numbers are $\mathds{N} = \{1,2,3, \ldots \}$ so $0$ is not included (I apologize to half the readers). For $n \in \mathds{N}$, let $[n]=\{1,\ldots,n\}$ and $[n]^{(k)} = \{s \subseteq [n]: |s|=k \}$. A partition of $n \in \mathds{N}$ is a sum $n=n_1 + \cdots + n_k$ with $n_i \in \mathds{N}$; in particular, each $n_i \geq 1$. Given a map $f: A^n \rightarrow A$ and $a \in A$, we write $f=\hat{a}$ to mean $f$ is the constant map which maps $A^n$ to $a$. Boldface font will be used to denote a vector or tuple $\textbf{a}=(a_1,\ldots,a_m) \in A^m$. For a binary relation $R \subseteq A \times A$, we will interchangeably use the equivalent notations $(a,b) \in R$, $a \ R \ b$ and $a \equiv_{R} b$ for membership in the relation. When $R$ is an equivalence relation and $a \in A$, then $a/R$ will denote the equivalence class containing $a$. For an algebra $\textbf{A}$, the largest congruence is the total relation $1_A \in \Con \textbf{A}$ and the smallest congruence is the identity relation $0_A \in \Con \textbf{A}$. We will sometimes omit the subscript when appropriate and just write $1$ and $0$ for the total and identity congruences, respectively.
 
For an algebra $\textbf{A}$ and congruence $\theta \in \Con A$, the \textit{lower central series} is defined by 
\[
(\theta]^{0}=\theta,\ \ \ (\theta]^{n+1} = [\theta,(\theta]^{n}]
\]
and the \textit{derived series} is 
\[
[\theta]^{0}=\theta,\ \ \ [\theta]^{n+1} = [\theta,[\theta]^{n}].
\]
A congruence $\theta \in \Con \textbf{A}$ is nilpotent of class n (equivalently, n-step nilpotent or n-nilpotent) if $(\theta]^n=0_A$, and solvable of class n if $[\theta]^n=0_A$. $\textbf{A}$ is nilpotent (solvable) if $1_A$ is nilpotent (solvable).

We sometimes use a convenient notation for evaluating operations on particular substitutions. Suppose we have an operation $f: A^n \rightarrow A$ and vector $\textbf{a} \in A^n$. Then for any $C \subseteq [n]$ and tuple $\textbf{b}=(b_1,\ldots,b_{|C|}) \in A^{|C|}$, $f(\textbf{a})_{C}[\textbf{b}]$ denotes the evaluation of $f$ where the elements of $\textbf{b}$ are substituted into the corresponding coordinates of $\textbf{a}$ specified by $C$. For example, we have 
\[
f(a_1,a_2,a_3,a_4,a_5,a_6,a_7)_{\{2,4,7\}}[b_1,b_2,b_3]=f(a_1,b_1,a_3,b_2,a_5,a_6,b_3).
\] 
For a specified list of variables $\{x_1,\ldots,x_n\}$, we will sometimes write a term $t(\bar{x}_S)$ for a subset $S \subseteq [n]$ to specify that the variables in $t$ are precisely $x_i$ for $i \in S$. This notation arises when discussing different subterms of a term.

A \textit{quasigroup} is an algebra $\textbf{A} = \left\langle A, \cdot, \backslash, / \right\rangle$ satisfying the identities
\begin{eqnarray*}
x \backslash (x \cdot y) &=& y \ \ \ \ (x \cdot y) / y  = x \\
x \cdot (x \backslash y) &=& y \ \ \ \ (x / y) \cdot y = x.
\end{eqnarray*}
In a quasigroup we always have the additional identities $(x / y) \backslash x = y = x / (y / x)$.
A \textit{loop} is a quasigroup which has a constant $0$ in the signature satisfying $0 \cdot x = x \cdot 0 = 0$. In a loop, we have the additional identities
\begin{eqnarray*}
0 \backslash x &=& x \ \ \ \ \ \ x \backslash x = 0 \\
x / 0 &=& x \ \ \ \ \ \ x / x = 0.
\end{eqnarray*}
Note any quasigroup has a Mal'cev term given by $m(x,y,z):= (x / (y \backslash y)) \cdot (y \backslash z)$.

For a fixed variety $\mathcal V$ in the signature $\sigma$, we say $\textbf{B} = \left\langle B, \sigma^{B}, F \right\rangle$ is an expanded $\mathcal V$-algebra if the reduct $\left\langle B, \sigma^{B} \right\rangle \in \mathcal V$ and $F$ is a set of operations.

\subsection{Nilpotent Mal'cev Algebras}\label{sec:nilpotent}

Let $\textbf{A}$ be an algebra in a Mal'cev variety. If $\textbf{A}$ is nilpotent, then $\textbf{A}$ is polynomially equivalent to an expanded loop \cite[Ch.7]{commod}. Let us describe in more detail how this is accomplished. Let $m(xyz)$ denote a Mal'cev term for $\textbf{A}$ and suppose $\textbf{A}$ is nilpotent of class $n$. Define a sequence of terms $f_n, g_n$ by
\begin{eqnarray*}
f_0(x,y,z) &:=& x \\
f_{n+1}(x,y,z) &:=& m(y,m(y,x,m(f_n(x,y,z),y,z)),f_n(x,y,z))
\end{eqnarray*}
and
\begin{eqnarray*}
g_0(x,y,z) &:=& z \\
g_{n+1}(x,y,z) &:=& m(g_n(x,y,z),m(m(x,y,g_{n}(x,y,z)),z,y),y). 
\end{eqnarray*}
The lower central series and the terms $f_n,g_n$ are related in $\textbf{A}$ in the following manner \cite[Lem.7.3]{commod}: for any $k \in \mathds{N}$ 
\begin{eqnarray*}
m(f_k(x,y,z),y,z) &(1]^k& x \ \ (1]^k \ \ m(z,y,g_k(z,y,x))\\
f_k(m(x,y,z),y,z) &(1]^k& x \ \ (1]^k \ \ g_k(z,y,m(z,y,x)).
\end{eqnarray*}
Since $\textbf{A}$ is nilpotent of class $n$, $(1]^n=0_A$; thus, for all $b,c \in A$ the maps $x \mapsto m(x,b,c)$ and $x \mapsto f_n(x,b,c)$ are inverses of each other, and the maps $x \mapsto m(c,b,x)$ and $x \mapsto g_n(c,b,x)$ are inverses of each other. It also follows that 
\begin{eqnarray*}
f_{n}(y,x,x) &=& y \ \ \ \ g_{n}(x,x,y) = y \\
f_{n}(x,y,x) &=& y \ \ \ \ g_{n}(x,y,x) = y
\end{eqnarray*}
in $\textbf{A}$. We can now define the polynomial loop reduct. Fix an element $0 \in A$ and define 
\[
x \cdot y := m(x,0,y) \ \ \ x\backslash y := g_{n}(x,0,y) \ \ \ x/y := f_{n}(x,0,y).
\]
Then $\textbf{A}$ is polynomially equivalent to the loop $L_\textbf{A} = \left\langle A, \cdot, \backslash , /, 0, F \right\rangle$ where we gather the remaining operations into $F$. $L_\textbf{A}$ is nilpotent of the same degree as $\textbf{A}$. In Section \ref{sec:charct}, we shall see how supernilpotence of $\textbf{A}$ is related to characterizations of which operations can be used for $F$ in order to generate $\Pol \textbf{A}$.

\subsection{Higher Commutators}

Let us now recall the definition of the $n$-commutator defined in Bulatov \cite{bulatov}.

\begin{definition}
Let $\textbf{A}$ be an algebra, and $\alpha_1,\ldots,\alpha_{n-1}, \beta, \delta \in \Con \textbf{A}$. We say that $\alpha_1,\ldots,\alpha_{n-1}$ centralizes $\beta$ modulo $\delta$, and write $\mathrm{C}(\alpha_1,\ldots,\alpha_{n-1},\beta;\delta)$ if for all polynomials $f(\textbf{x}_1,\ldots,\textbf{x}_{n-1},\textbf{y})$ and vectors $\textbf{a}_1, \textbf{b}_1,\ldots, \textbf{a}_{n-1}, \textbf{b}_{n-1}, \textbf{c}, \textbf{d}$ satisfying
\begin{enumerate}

\item $\textbf{a}_i \equiv_{\alpha_i} \textbf{b}_i$ for $i=1,\ldots,n-1$,

\item $\textbf{c} \equiv_{\beta} \textbf{d}$, and

\item $f(\textbf{z}_1,\ldots,\textbf{z}_{n-1},\textbf{c}) \equiv_{\delta} f(\textbf{z}_1,\ldots,\textbf{z}_{n-1},\textbf{d})$ for every tuple $(\textbf{z}_1,\ldots,\textbf{z}_{n-1}) \in \left\{\textbf{a}_1, \textbf{b}_1\right\} \times \cdots \times \left\{\textbf{a}_{n-1},\textbf{b}_{n-1}\right\} \backslash \left\{(\textbf{b}_1,\ldots,\textbf{b}_{n-1})\right\}$,

\end{enumerate}
we have $f(\textbf{b}_1,\ldots,\textbf{b}_{n-1},\textbf{c}) \equiv_{\delta} f(\textbf{b}_1,\ldots,\textbf{b}_{n-1},\textbf{d})$.
\end{definition}

In this case we write $\textbf{A} \vDash \mathrm{C}(\alpha_1,\ldots,\alpha_{n-1},\beta;\delta)$. The following properties follow directly by application of the definition of the centralizer relation.

\begin{lemma}\label{lem:cen}
We have the following properties of the centralizer relation in an algebra $\textbf{A}$:
\begin{enumerate}

	\item $\textbf{A} \vDash \left(\bigwedge_{t \in I} \mathrm{C}(\alpha_1,\ldots,\alpha_n;\delta_t)\right) \rightarrow \mathrm{C}(\alpha_1,\ldots,\alpha_n;\bigwedge_{t \in I} \delta_t)$
	
	\item If each $\alpha_i \leq \beta_i$, then $\textbf{A} \vDash \mathrm{C}(\beta_1,\ldots,\beta_n;\delta) \rightarrow \mathrm{C}(\alpha_1,\ldots,\alpha_n;\delta)$.
	
	\item For any permutation $\sigma: [n-1] \rightarrow [n-1]$, 
	\[
	\textbf{A} \vDash \mathrm{C}(\alpha_1,\ldots,\alpha_{n-1},\alpha_{n};\delta) \,\,\, \mathrm{iff} \,\,\, \textbf{A} \vDash \mathrm{C}(\alpha_{\sigma(1)},\ldots,\alpha_{\sigma(n-1)},\alpha_{n};\delta).
	\]
	
	\item For $\gamma \leq \alpha_1 \wedge \cdots \wedge \alpha_n \wedge \delta$, 
	\[
	\textbf{A} \vDash \mathrm{C}(\alpha_1,\ldots,\alpha_n;\delta) \,\,\, \mathrm{iff} \,\,\, \textbf{A}/\gamma \vDash \mathrm{C}(\alpha_1/\gamma,\ldots,\alpha_n/\gamma,\beta/\gamma;\delta/\gamma).
	\]
	
	\item For $i < n$, $\textbf{A} \vDash \mathrm{C}(\alpha_1,\ldots,\alpha_{i-1},\alpha_{i+1},\ldots,\alpha_n;\delta) \rightarrow \mathrm{C}(\alpha_1,\ldots,\alpha_n;\delta)$.
	
	\item If $\alpha_n \wedge (\beta \circ (\delta \wedge \alpha_n) \circ \beta) \subseteq \delta$, then for $k < n$ we have
	\[
	\textbf{A} \vDash \mathrm{C}(\alpha_1,\ldots,\alpha_{k-1},\beta,\alpha_{k+1},\ldots,\alpha_n;\delta);
	\]
if in addition $\alpha_n \wedge (\gamma \circ (\delta \wedge \alpha_n) \circ \gamma) \subseteq \delta$, then we have 
\[
\textbf{A} \vDash \mathrm{C}(\alpha_1,\ldots,\alpha_{k-1},\beta \vee \gamma,\alpha_{k+1},\ldots,\alpha_n;\delta).
\]

\end{enumerate}
\end{lemma}
\vspace{0.1cm}

Property (1) in Lemma \ref{lem:cen} yields the following definition.

\begin{definition}
For an algebra $\textbf{A}$, we let $[\alpha_1,\ldots,\alpha_{n}]$ be the smallest congruence $\delta$ such that $\textbf{A} \vDash \mathrm{C}(\alpha_1,\ldots,\alpha_n;\delta)$.
\end{definition}

For any arbitrary algebra $\textbf{A}$ we have the the following properties:
\vspace{0.1cm}

(HC1) $[\alpha_1,\ldots,\alpha_n] \leq \bigwedge_{1 \leq i \leq n} \alpha_i$;

(HC2) If each $\alpha_i \leq \beta_i$, then $[\alpha_1,\ldots,\alpha_n] \leq [\beta_1,\ldots,\beta_n]$;

(HC3) $[\alpha_1,\ldots,\alpha_n] \leq [\alpha_2,\ldots,\alpha_n]$.
\vspace{0.1cm}

If $\textbf{A}$ generates a Mal'cev variety, then the following hold \cite{supernil}:
\vspace{0.1cm}

(HC4) For any permutation $\sigma: [n] \rightarrow [n]$, 
\[
[\alpha_1,\ldots,\alpha_n]=[\alpha_{\sigma(1)},\ldots,\alpha_{\sigma(n-1)}];
\]

(HC5) $[\alpha_1,\ldots,\alpha_n] \leq \eta$ if and only if $\textbf{A} \vDash \mathrm{C}(\alpha_1,\ldots,\alpha_n;\eta)$;

(HC6) If $\eta \leq \alpha_1 \wedge \cdots \wedge \alpha_n$, then in $\textbf{A}/\eta$ we have 
\[
[\alpha_1/\eta,\ldots,\alpha_n/\eta] = \left( [\alpha_1,\ldots,\alpha_n] \vee \eta \right) /\eta;
\]

(HC7) For any $j \in [n]$ and collection $\{\beta_i\}_{i \in I} \subseteq \Con \textbf{A}$, 
\[
[\alpha_1,\ldots,\alpha_{j-1},\bigvee_{i \in I} \beta_{i}, \alpha_{j+1},\ldots,\alpha_n] = \bigvee_{i \in I} [\alpha_1,\ldots,\alpha_{j-1},\beta_{i}, \alpha_{j+1},\ldots,\alpha_n];
\]

(HC8) For any $1 < i < n$, $[\alpha_1,\ldots,\alpha_{i-1},[\alpha_i,\ldots,\alpha_n]] \leq [\alpha_1,\ldots,\alpha_n]$.
\vspace{0.2cm}

For the higher commutator, it is useful to adopt the notation $[\theta,\ldots,\theta]_n = [\underbrace{\theta,\ldots,\theta}_{n}]$ for $\theta \in \Con \textbf{A}$. A congruence $\alpha \in \Con \textbf{A}$ is $n$-\textit{supernilpotent} if $[\alpha,\ldots,\alpha]_n=0_A$. The algebra $\textbf{A}$ is said to be $n$-supernilpotent if $1_A \in \Con \textbf{A}$ is $n$-supernilpotent. The terminology is justified by the following observation: if $\textbf{A}$ generates a Mal'cev variety, then by repeated applications of (HC8) for any $\theta \in \Con \textbf{A}$ we have
\begin{eqnarray*}
[\theta,\ldots,\theta]_n \geq [\theta,[\theta,\ldots,\theta]_{n-1}] &\geq& [\theta,[\theta,[\theta,\ldots,\theta]_{n-2}]] \\
&\vdots& \\
&\geq& [\theta,[\theta, \cdots[\theta,\theta] \cdots ]].
\end{eqnarray*}
with n-1 nested binary commutators in the last term. If $\textbf{A}$ is $n$-supernilpotent, then $\textbf{A}$ is nilpotent of class n-1.

The characterization of $3$-supernilpotent Mal'cev algebras we wish to extend is the following:

\begin{theorem}(\cite{mudrinski})\label{thm:mudrinski}
Suppose $\textbf{A}$ generates a Mal'cev variety. The following are equivalent:
\begin{enumerate}

	\item $\textbf{A}$ is 3-supernilpotent ($[1_A,1_A,1_A]=0_A$).
	
	\item $\textbf{A}$ is polynomially equivalent to an expanded group $\textbf{V} = \left\langle A, +, -, 0, F \right\rangle$ such that
	
		\begin{enumerate}
		
			\item $F$ is the set of at most binary absorbing polynomials of $\textbf{V}$,
			
			\item Every absorbing binary polynomial is distributive with respect to $+$ in both arguments, and   
			
			\item $\textbf{V}$ is 2-nilpotent ($[1_V,[1_V,1_V]]=0_V$).
		
		\end{enumerate}

\end{enumerate}
\end{theorem}

\section{Supernilpotent Mal'cev Algebras}\label{sec:charct}

\begin{definition}\cite{supernil}\label{def:absorb}
Let $f(x_1,\ldots,x_n) \in \Pol_n \textbf{A}$ and $a_1,\ldots,a_n,a \in A$. We say $f$ absorbs $(a_1,\ldots,a_n)$ to $a$ if 
$f(b_1,\ldots,b_n)=a$ whenever some $b_i=a_i$.
\end{definition}

If $a \in A$, we say $f(x_1,\ldots,x_n)$ is $a$-absorbing if $f$ absorbs $(a,\ldots,a)$ to $a$. Let $\Ab_{a}^k \textbf{A}$ denote the set of $k$-ary $a$-absorbing polynomials of $\textbf{A}$.

Absorbing polynomials interact with the commutator operation on the congruence lattice in ways which mimic the behavior of the commutator term $[x,y]:=xyx^{-1}y^{-1}$ in relation to a normal series in a group $\textbf{G}$. Note $[x,y] \in \Ab^{2}_{\id} \textbf{G}$.

\begin{lemma}\label{lem:absorb}
Let $\textbf{A}$ be an algebra. If $f \in \Pol_k \textbf{A}$ absorbs $(a_1,\ldots,a_k)$ to $a$, then for any partition $n=n_1 + \cdots + n_k$ and $\{\alpha^{i}_{j}:i \in [k],j \in [n_i]\} \subseteq \Con \textbf{A}$ we have 
\[
f(a_1/[\alpha^{1}_{1},\ldots,\alpha^{1}_{n_1}],\ldots,a_k/[\alpha^{k}_{1},\ldots,\alpha^{k}_{n_k}]) \subseteq a/\left[[\alpha^{1}_{1},\ldots,\alpha^{1}_{n_1}],\ldots,[\alpha^{k}_{1},\ldots,\alpha^{k}_{n_k}]\right]
\]
where we set $[\alpha]=\alpha$; in particular, if $\textbf{A}$ is $n$-supernilpotent, then $\Ab^{m}_{a} \textbf{A}=\{\hat{a}\}$ for $m \geq n$.
\end{lemma}
\begin{proof}
Let $b_i \in a_i/[\alpha^{i}_{1},\ldots,\alpha^{i}_{n_i}]$ for $i \in [k]$. Then for any tuple $\bar{u} \in  \left\{\bar{a}_1, \bar{b}_1\right\} \times \cdots \times \left\{\bar{a}_{k-1},\bar{b}_{k-1}\right\} \backslash \left\{(\bar{b}_1,\ldots,\bar{b}_{k-1})\right\}$ we see that $f(\bar{u},a_k)=a=f(\bar{u},b_k)$ since some coordinate of $\bar{u}$ must be an $a_i$. This yields 
\[
a=f(b_1,\ldots,b_{k-1},a_k) \left[[\alpha^{1}_{1},\ldots,\alpha^{1}_{n_1}],\ldots,[\alpha^{k}_{1},\ldots,\alpha^{k}_{n_k}]\right] f(b_1,\ldots,b_{k-1},b_k). 
\]
Now assume $\textbf{A}$ is $n$-supernilpotent. For any $f \in \Ab^{m}_{a} \textbf{A}$ and $b_1,\ldots,b_m \in A=a/1_A$, we have $(f(b_1,\ldots,b_m),a) \in [1,\ldots,1]_m \leq [1,\ldots,1]_n=0_A$ by the first part and (HC3). 
\end{proof}

\begin{definition}\label{def:ntype}
An expanded loop $\textbf{A}=\left\langle A , \cdot , \backslash , / , 0 , F \right\rangle$ is said to be $n$-\textit{type} if
\begin{enumerate}

	\item $F$ is the set of at most (n-1)-ary absorbing polynomials of $\textbf{A}$; 
	
	\item every $f \in \Ab_{0}^{n-1} \textbf{A}$ distributes over ``$\cdot$'' in each coordinate;
	
	\item the n-commutator satisfies $[1,\ldots,1,[1,\ldots,1]_{k}]=0_A$ for each $2 \leq k < n$. 

\end{enumerate}
\end{definition}

From condition (3) above, we see that a $n$-type loop must be nilpotent of class n-1 by repeated application of (HC8). The next two technical lemmas will be used in the proof of Theorem \ref{thm:nabelian}.

\begin{lemma}\label{lem:distr}
Let $\textbf{A}=\left\langle A , \cdot , \backslash , / , 0 , F \right\rangle$ be a $n$-type expanded loop and fix $2 \leq k < n-1$. Suppose we have 
\begin{itemize}

	\item $f \in \Ab^{k}_{0} \textbf{A}$;
	
	\item a partition $n=n_1 + \cdots + n_k$ with $2 \geq n_1$;

	\item a parameter $1 \leq p \leq k$ with $a_{p+1},\ldots,a_{k} \in A$ and polynomials $g_i \in \Ab^{m_i}_{0} A$ with $m_i \geq n_i$ for $1 \leq i \leq p$.

\end{itemize}
Then $f(g_1,\ldots,g_p,a_{p+1},\ldots,a_{k})=\hat{0}$. 
\end{lemma}
\begin{proof}
By Definition \ref{def:ntype}(3) we have $[1,\ldots,1,[1,\ldots,1]_{n_1}]_{n-n_1}=0$ since $n_1 \geq 2$. Repeated applications of (HC8) and (HC4) yields
\begin{equation}\label{eq:lem3.4}
[[1,\ldots,1]_{n_1},\ldots,[1,\ldots,1]_{n_k}]=0. 
\end{equation}
Lemma \ref{lem:absorb} implies each $(g_i(b_1,\ldots,b_{m_i}),0) \in [1,\ldots,1]_{m_i} \leq [1,\ldots,1]_{n_i}$ for all $b_1,\ldots,b_{m_i} \in A$. We always have $(a_i,0) \in 1_A$ and $(h(b),0) \in 1_A$ for any $a_i,b \in A, b, h \in \Ab^{1}_{0} \textbf{A}$. Then another application of Lemma \ref{lem:absorb} and (\ref{eq:lem3.4}) finishes the calculation.
\end{proof}

\begin{lemma}\label{lem:closure}
Let $L = \textbf{A}=\left\langle A , \cdot , \backslash , / , 0 , F \right\rangle$ be a $n$-type expanded loop. If $f \in \Ab^{k}_{0} \textbf{L}$ with $1 \leq k \leq n-1$, then for any left-associated product $\prod_{j=1}^{m} y_i$ we have that 
\[
f(x_1,\ldots,x_{i-1},\prod_{j=1}^{m} y_j,x_{i+1},\ldots,x_k)
\] 
is a product of absorbing polynomials in the $x_i$'s and $y_j$'s of arities $p$ where $k \leq p \leq n-1$.
\end{lemma}
\begin{proof}
The proof is by descent from $n-1$ down to $1$. By permuting the coordinates, it is sufficient to establish the lemma for the first coordinate. 
For the base case, take $g \in \Ab^{n-1}_{0} \textbf{L}$. By condition \ref{def:ntype}(2), $g$ distributes over $\cdot$ so writing $\bar{x}=(x_2,\ldots,x_{n-1})$ we have $g(\prod_{i=1}^{m} y_i,\bar{x}) = \prod_{i=1}^{m} g(y_i,\bar{x})$. Now assume the result for $k+1$. Take $f \in \Ab^{k}_{0} \textbf{L}$ and define
\[
f'(y_1,y_2,x_2,\ldots,x_k):= f(y_2,x_2,\ldots,x_k) \backslash \left( f(y_1,x_2,\ldots,x_k) \backslash f(y_1 + y_2,x_2,\ldots,x_k)\right).
\]
Then $f' \in \Ab^{k+1}_{0} \textbf{L}$. Solving for $f$ we have 
\begin{eqnarray*}
f(\prod_{i=1}^{m} y_i,\bar{x}) &=& f(\prod_{i=1}^{m-1} y_i,\bar{x}) \left( f(y_m,\bar{x}) \cdot f'(\prod_{i=1}^{m-1} y_i,y_m,\bar{x}) \right) \\
&=& \left( \cdots \left((f(y_1,\bar{x}) \cdot s_2 ) \cdot s_3 \right) \cdots \right) \cdot s_m
\end{eqnarray*}
where $s_j = f(y_j,\bar{x}) \cdot f'(\prod_{i=1}^{j-1},y_j,\bar{x})$. By the inductive hypothesis, each of the polynomials $f'(\prod_{i=1}^{j-1},y_j,\bar{x})$ is a product of absorbing polynomials of arities $p$ where $k+1 \leq p \leq n-1$; therefore, each $s_j$, and so $f(\prod_{i=1}^{m} y_i,\bar{x})$, is a product of absorbing polynomials of arities $p$ where $k \leq p \leq n-1$. 
\end{proof}

The next definition is meant to give a convenient generalization of when an algebra is polynomially equivalent to a module over a unital ring.

\begin{definition}\label{def:rep}
Let $\textbf{A}$ be an algebra. We say $\textbf{A}$ has a \textit{representation of degree} $n$ if the there exists $0 \in A$ and a binary polynomial $x \cdot y:= r(x,y) \in \Pol_2 \textbf{A}$ such that every polynomial $f(x_1,\ldots,x_k) \in \Pol_k \textbf{A}$ can be written as a left-associated product
\[
f(x_1,\ldots,x_k) = \prod_{i=0}^{p} r_i = \left(\left( \cdots \left( \left(c \cdot r_1 \right) \cdot r_2 \right) \cdots  \right)\cdot r_{p-1}\right) \cdot r_p 
\]
where
\begin{enumerate}

	\item $r_0 = c \in A$;

	\item $p=\min\{k,n\}$;

	\item $0 \cdot x = x \cdot 0 =  x$ for all $x \in A$;
	
	\item for $1 \leq m \leq p$, each $r_m=r_m(x_1,\ldots,x_k)$ is a left-associated product 
	\[
	r_m(x_1,\ldots,x_k) =  \prod_{S \in [k]^{(m)}} t^{m}_{S}(\bar{x}_{S})
	\]
where $t^{m}_{S}(\bar{x}_{S}) \in \Ab^{m}_{0} \textbf{A}$.
\end{enumerate} 
\end{definition}

Since the binary operation is not associative in general, the order of the parentheses matter.

\begin{example}
Let $\textbf{R}$ be a ring and $\textbf{M}$ an $\textbf{R}$-module. Any polynomial $p(x_1,\ldots,x_n)$ in $\textbf{M}$ has the form $p(x_1,\ldots,x_n) = c + r_1x_1 + \cdots + r_nx_n$ for elements $c \in M$, $r_1,\ldots,r_n \in R$; thus, $\textbf{M}$ has a representation of degree 1.
\end{example}

\begin{example}
Let us recall how to deduce the module structure of an abelian algebra $\textbf{A}$ with a Mal'cev term $m(x,y,z)$ \cite{herrmann,commod}.  Fix $0 \in A$ and define $x+y:=m(x,0,y)$. Since $\textbf{A}$ is an abelian Mal'cev algebra, one uses the centralizer relation to show ``+'' is associative and commutative with an additive inverse given by $-x:=m(0,x,0)$; thus, $\left\langle A, +, -x, 0 \right\rangle$ is an abelian group. For each $f \in \Pol_k A$, define 
\[
r_i(x) := m(f(0,\ldots,0,x,0,\ldots,0),f(0,\ldots,0),0)
\] 
where $x$ is in the $i$-th coordinate of $f$. Then $r_i \in \Ab^{1}_{0} \textbf{A}$. Using the centralizer relation one shows that $f(x_1,\ldots,x_k) = f(0,\ldots,0) + \sum_{i=1}^{k} r_i(x_i)$ which yields a representation of degree 1.

Let $R=\Ab^{1}_{0} \textbf{A}$ and consider the operation $+$ induced on $R$. It follows that $\left\langle R, + , -, 0 \right\rangle$ is an abelian group. If we define a multiplication on $R$ by composition $\circ$, then $\textbf{R} = \left\langle R, \circ, +, -, 0, \id \right\rangle$ is a ring with unity given by the identity map $\id$. If we define an action of $R$ on $A$ by $r \ast a = r(a)$ for $r \in R$, $a \in A$, then another argument using the centralizer relation proves this determines a module representation of $\textbf{R}$ which we denote by $\textbf{M}_A$. The degree 1 representation shows $\textbf{A}$ is polynomially equivalent to the $\textbf{R}$-module $\textbf{M}_A$ in which each $f \in \Pol_k \textbf{A}$ is explicitly represented as $f(x_1,\ldots,x_k) = c + \sum_{i=1}^{k} r_i\cdot x_i$.
\end{example}

\begin{example}
Let $\textbf{R}$ be a ring. The congruences of $\textbf{R}$ are exactly the partitions induced by the cosets of left-right-ideals (referred to as ideals in this example). If $I$ is an ideal, let $\alpha_I$ denote the congruence induced by $I$. If $I_1,\ldots,I_n$ are ideals of $\textbf{R}$, then it follows from \cite[Cor 6.12]{supernil} or explicitly from \cite[Lem 3.5]{mayr} that $[\alpha_{I_1},\ldots,\alpha_{I_n}] = \alpha_{K}$ where $K=\sum_{\sigma \in S_n} I_{\sigma(1)} \cdots I_{\sigma(n)}$ and $S_n$ is the permutation on $n$ letters. We see that a ring is $n$-supernilpotent if and only if all products of $n$ elements are trivial; for example, a ring is abelian if and only if it has trivial multiplication. If we take $\textbf{F}$ to be the free ring in a single generator $\{x\}$ without identity and the ideal $I=Fx^{n-1}$, then $\textbf{F}/I$ is n-supernilpotent. Supernilpotent rings may have non-trivial multiplication but cannot have a multiplicative identity. 
\end{example}

\begin{theorem}\label{thm:nabelian}
Let $\textbf{A}$ be an algebra in a Mal'cev variety. The following are equivalent:
\begin{enumerate}

	\item $\textbf{A}$ is $n$-supernilpotent;
	
	\item $\textbf{A}$ is nilpotent and has a representation of degree n-1;
	
	\item $\textbf{A}$ is polynomially equivalent to an $n$-type loop.
	
\end{enumerate}	
\end{theorem}
\begin{proof}
Let $m(x,y,z)$ be a Mal'cev term for $\textbf{A}$.

$(1) \Rightarrow (2)$: Assume $\textbf{A}$ is $n$-supernilpotent. From the discussion in Section \ref{sec:prelim}, $\textbf{A}$ is nilpotent of class n-1 and polynomially equivalent to a loop $\textbf{L}_A = \left\langle A , \cdot , \backslash , / , 0 , F \right\rangle$ which is also nilpotent of class n-1. We will establish the representation by a form of polynomial interpolation in $L_\textbf{A}$. We take the loop multiplication ``$\cdot$'' as the basic binary product and note that condition \ref{def:rep}(3) is immediately satisfied.

Take $f \in \Pol_m \textbf{A}$. Set $f(0,\ldots,0)=c$. Inductively, we shall define for $0 \leq k \leq m$ polynomials $r_k(x_1,\ldots,x_m)$ such that
\begin{equation}\label{eq:thm3.10}
\left(\left(\left( \cdots \left( \left(c \cdot r_1 \right) \cdot r_2 \right) \cdots  \right)\cdot r_{k-1}\right) \cdot r_k\right)  (\bar{0})_{S}[\bar{x}_S] =f(\bar{0})_{S}[\bar{x}_S]
\end{equation}
for each $S \in [m]^{(k)}$ where $\bar{0}=(0,\ldots,0)$. Set $r_0(\bar{x}) :=  f(0,\ldots,0)=c \in A$. For $i \in [m]$, define 
\[
t_i(x):= c\backslash f(\bar{0})_{\{i\}}[x]. 
\] 
Then $t_i(0) = c\backslash f(\bar{0})_{\{i\}}[0]= c \backslash f(\bar{0}) = c \backslash c = 0$; thus, $t_i \in \Ab^{1}_{0} \textbf{A}$. Then set $r_1(x_1,\ldots,x_m):= \prod_{i \in [m]} t_i(x_i)$ a left-associated product. We then check (\ref{eq:thm3.10}): 
\[
\left(c \cdot r(x_1,\ldots,x_m)\right) (\bar{0})_{\{i\}}[x_i] = c \cdot t_i(x_i) = c \cdot ( c\backslash f(\bar{0})_{\{i\}}[x_i]) = f(\bar{0})_{\{i\}}[x_i].
\] 
Inductively, suppose $r_k(x_1,\ldots,x_m)$ is defined for $k < m$ and satisfies (\ref{eq:thm3.10}). Now for $S \in [m]^{(k+1)}$, define
\[
t_{S}(\bar{x}_S):= \left(( \prod_{i=0}^{k} r_i )(\bar{0})_{S}[\bar{x}_{S}] \right) \, \backslash \, f(\bar{0})_{S}[\bar{x}_S].
\] 
Then for any $i \in S$, 
\begin{eqnarray*}
t_{S}(\bar{x}_S)_{i}[0] &=& \left(( \prod_{i=0}^{k} r_i )(\bar{0})_{S \backslash \{i\}}[\bar{x}_{S \backslash \{i\}}] \right) \, \backslash \, f(\bar{0})_{S \backslash \{i\}}[\bar{x}_{S \backslash \{i\}}] \\
 &=& f(\bar{0})_{S \backslash \{i\}}[\bar{x}_{S \backslash \{i\}}] \, \backslash \, f(\bar{0})_{S \backslash \{i\}}[\bar{x}_{S \backslash \{i\}}] = 0
\end{eqnarray*}
using the inductive hypothesis. This shows $t_{S}(\bar{x}_S) \in \Ab^{k+1}_{0} \textbf{A}$. Define the left-associated product $r_{k+1}(x_1,\ldots,x_m) := \prod_{S \in [m]^{k+1}} t_{S}(\bar{x}_S)$ and verify (\ref{eq:thm3.10}): fix $T \in [m]^{(k+1)}$ and calculate
\begin{eqnarray*}
(\prod_{i=0}^{k+1} r_i )(\bar{0})_{T}[\bar{x}_{T}] &=&  \left(( \prod_{i=0}^{k} r_i )(\bar{0})_{T}[\bar{x}_{T}] \right) \cdot t_{T}(\bar{x}_T) \\
&=&  \left(( \prod_{i=0}^{k} r_i )(\bar{0})_{T}[\bar{x}_{T}] \right) \cdot \left(\left(( \prod_{i=0}^{k} r_i )(\bar{0})_{T}[\bar{x}_{T}] \right) \, \backslash \, f(\bar{0})_{T}[\bar{x}_T] \right)\\
&=& f(\bar{0})_{T}[\bar{x}_T].
\end{eqnarray*}

Now for $k=m$, note $r_i(\bar{0})_{[m]}[\bar{x}_{[m]}]= r_i(x_1,\ldots,x_m)$ when $i \leq m$. So by (\ref{eq:thm3.10}) we have for the left-associated product
\[
\prod_{i=0}^{m} r_i(x_1,\ldots,x_m) \hspace{8.3cm} 
\]
\begin{eqnarray*}
&=& \left(\prod_{i=0}^{m-1} r_i(x_1,\ldots,x_m) \right) \cdot \left(\left( \prod_{i=0}^{m-1} r_i(x_1,\ldots,x_m) \right)\backslash f(x_1,\ldots,x_m) \right) \\
&=& f(x_1,\ldots,x_m).
\end{eqnarray*}
Since $\textbf{A}$ is $n$-supernilpotent, by Lemma \ref{lem:absorb} $t_S(\bar{x}_S)=\hat{0}$ for $S \in [m]^{(k)}$ with $k \geq n$; thus, $r_k(x_1,\ldots,x_m)=\hat{0}$ for $k \geq n$. This reduces the representation above to
\[
f(x_1,\ldots,x_m) = \prod_{i=0}^{p} r_i(x_1,\ldots,x_m)
\]
where $p = \min\{m,n-1\}$. Altogether, we have established that $\textbf{A}$ has a representation of degree n-1.

$(1) \Rightarrow (3)$: Assume $\textbf{A}$ is $n$-supernilpotent. Again we have that $\textbf{A}$ is polynomially equivalent to a loop $\textbf{L}_A = \left\langle A , \cdot , \backslash , / , 0 , F \right\rangle$ which is nilpotent of class n-1. Condition \ref{def:ntype}(3) is established by using (HC8). To show \ref{def:ntype}(2), let $f \in \Ab^{n-1}_{0} \textbf{A}$ and define 
\[
f'(y,z,x_2,\ldots,x_{n-1}) : = \hspace{8cm}
\]
\[
\left(f(y \cdot z,x_2,\ldots,x_{n-1})/ f(z,x_2,\ldots,x_{n-1})\right)/ f(y,x_2,\ldots,x_{n-1}).  
\]
Then $f' \in \Ab^{n}_{0} \textbf{A}$ and so by Lemma \ref{lem:absorb}, $f'=\hat{0}$; thus, 
\[
f(y \cdot z,x_2,\ldots,x_{n-1}) = f(y,x_2,\ldots,x_{n-1}) \cdot f(z,x_2,\ldots,x_{n-1}).
\]
Inductively, we have for a left associated product $\prod_{i=1}^{m} y_i$,
\begin{eqnarray*}
f(\prod_{i=1}^{m} y_i,x_2,\ldots,x_{n-1}) &=& f(\prod_{i=1}^{m-1} y_i,x_2,\ldots,x_{n-1}) \cdot f(y_m,x_2,\ldots,x_{n-1}) \\
&=& \prod_{i=1}^{m} f(y_i,x_2,\ldots,x_{n-1}) 
\end{eqnarray*}
which is again left-associated. The same argument works for each coordinate. This establishes \ref{def:ntype}(2). From the implication $(1) \Rightarrow (2)$ above, we have a representation of degree n-1 for $\textbf{A}$. This implies we can take $F' = \bigcup_{k=1}^{n-1} \Ab^{k}_{0} \textbf{A} \subseteq F$ in order to generate $\Pol \textbf{A}$ and so condition \ref{def:ntype}(1) is satisfied.

$(3) \Rightarrow (2)$: Assume $\textbf{A}$ is polynomially equivalent to the $n$-type loop $\textbf{L} = \left\langle A , \cdot , \backslash , / , 0 , F \right\rangle$. So $\Pol_k \textbf{A}= \Pol_k \textbf{L}$ for all $k \in \mathds{N}$ and both $\textbf{A}$ and $\textbf{L}$ are nilpotent of class n-1. Let $P_k$ denote the set of polynomials $f(x_1,\ldots,x_k) \in \Pol_k \textbf{L}$ which can be represented in the form of Definition \ref{def:rep} for n-1; clearly, $P_k \subseteq \Pol_k \textbf{L}$. We shall establish $\Pol_k \textbf{L} \subseteq P_k$ by showing that $P_k$ contains the constants, projections and is closed under the fundamental operations of $\textbf{L}$. That $P_k$ contains the constants and projections is immediate. It is also clear that $F \cap \Pol_k \textbf{L} \subseteq P_k$ for $k=1,\ldots,n-1$. For the binary operations $\cdot,\backslash,/$ the proof in $(1) \Rightarrow (2)$ works again here to allow us to ``solve'' for the representations. The more involved argument is closure under $F$.

Fix $f \in \Ab^{m}_{0} \textbf{L} \subseteq F$. Take $g_1,\ldots,g_k \in P_k$. We need to show $f(g_1,\ldots,g_m) \in P_k$. Since each $g_i \in P_k$, we can write a left-associated product $g_i = \prod_{j=0}^{n-1} r^{i}_{j}$ for $i=1,\ldots,m$. By Lemma \ref{lem:closure}, $f(g_1,\ldots,g_m)$ is a product (associated in some manner) of polynomials $h(\xi_1,\ldots,\xi_p)$ where $h \in \Ab^{p}_{0} \textbf{L}$ for $m \leq p \leq n-1$ and $\{\xi_1,\ldots,\xi_p\} \subseteq \{r^{i}_{j} : i \in [m], 0 \leq j \leq n-1 \}$. Another application of Lemma \ref{lem:closure} allows us to write each $h(\xi_1,\ldots,\xi_p)$ as a product (associated in some manner) of absorbing polynomials $h'(\zeta_1,\ldots,\zeta_q)$ with $h' \in \Ab^{q}_{0} \textbf{L}$ for $m \leq p \leq q \leq n-1$ and $\{\zeta_1,\ldots,\zeta_q\} \subseteq \{t^{i}_{S}(\bar{x}_S): i \in [m], S \in [k]^{j},0 \leq j \leq n-1 \}$ where the last set comprises all the absorbing polynomials which appear in the products representing the $r^{i}_{j}$'s. 

Since each $\zeta_i \in \Ab^{j_i}_{0} \textbf{L}$, we have $h'(\zeta_1,\ldots,\zeta_q) \in \Ab^{\sigma}_{0} \textbf{L}$ for some $\sigma \leq k^{\ast}=j_{1} + \cdots + j_{q}$ since the $\zeta_i$ may have variables in common. According to Lemma \ref{lem:distr}, we see that $h'(\zeta_1,\ldots,\zeta_q)=\hat{0}$ unless $k^{\ast} \leq n-1$. This just means each non-trivial $h'(\zeta_1,\ldots,\zeta_q) \in F$. Since $P_k$ is closed under $\cdot$ and we have shown $f(g_1,\ldots,g_m)$ is a finite product of absorbing polynomials with representations in $P_k$, it follows that $P_k$ is closed under the operations in $F$.

$(2) \Rightarrow (1)$: Assume $\textbf{A}$ is nilpotent and has a representation of degree $n-1$. We first show $\Ab^{k}_{0} \textbf{A}= \{\hat{0}\}$ where $k \geq n$. Let $f \in \Ab^{k}_{0} \textbf{A}$. Then we can write 
\[
f(x_1,\ldots,x_k) = \left(\left( \cdots \left( \left(c \cdot r_1 \right) \cdot r_2 \right) \cdots  \right)\cdot r_{n-2}\right) \cdot r_{n-1} 
\]
where $r_{m}(x_1,\ldots,x_k) = \prod_{s \in [k]^{(m)}} t^{m}_{s}(\bar{x}_{s})$ is a left-associated product with $t^{m}_{s}(\bar{x}_{s}) \in \Ab^{m}_{0} \textbf{A}$. Using that $f$ is $0$-absorbing, we see that $0=f(0,\ldots,0)=c$, and so $0=f(\bar{x})_{[n]\backslash \{i\}}[0] = t_{i}(x_i)$ for each $i \in [n]$. This shows $r_1(x_1,\ldots,x_k)=\hat{0}$. Continuing in this fashion, having established $r_m(x_1,\ldots,x_m)=\hat{0}$ we see that $0=f(\bar{x})_{[n] \backslash s}[0,\ldots,0] = t^{m+1}_{s}(\bar{x}_{s})$ for each $s \in [k]^{(m+1)}$ and so we have $r_{m+1}(x_1,\ldots,x_k)=\hat{0}$. Inductively, we have shown $f=\hat{0}$. 

Now let $g \in \Pol_n \textbf{A}$ and suppose $g$ absorbs $(a_1,\ldots,a_n)$ to $a$. We will show $g=\hat{a}$. It will then follow from \cite[Prop.6.16]{supernil} that the $n$-commutator $[1,\ldots,1]=0_A$ and so $\textbf{A}$ is $n$-supernilpotent. Since $\textbf{A}$ is nilpotent, for each $i \in [n]$ the unary polynomial $h_i(x):=m(x,0,a_i)$ is bijective. Define the polynomial 
\[
t(x_1,\ldots,x_n):=m(g(h_1(x_1),\ldots,h_n(x_n)),a,0).
\]
Then $t \in \Ab^{n}_{0} \textbf{A}$ and so $t=\hat{0}$ be the previous paragraph. Again, the unary polynomial $m(x,a,0)$ is bijective by nilpotence. Then for all $b_1,\ldots,b_n \in A$, the evaluation 
\[
m(g(h_1(b_1),\ldots,h_n(b_n)),a,0)=t(b_1,\ldots,b_n)=0=m(a,a,0)
\] 
implies $g(h_1(b_1),\ldots,h_n(b_n))=a$. So $g(h_1(x_1),\ldots,h_n(x_n))=\hat{a}$. Since each $h_i$ is bijective it must be that $g=\hat{a}$.
\end{proof}

The argument in the last paragraph can be used to simplify the generating set for the higher commutators in nilpotent algebras. This can be seen as complimenting the fact that nilpotent algebras in varieties with a weak difference term have \textit{regular} congruences (consult \cite[Thm 4.8]{twocommutator}, \cite[Cor 7.7]{commod}); that is, congruences as equivalence relations are uniquely determined by any one of their equivalence classes.

\begin{lemma}\label{lem:cong}
Let $\textbf{A}$ be a nilpotent algebra in a variety with a weak difference term. Fix $0 \in A$. For any $\theta_1, \ldots, \theta_n \in \Con \textbf{A}$, 
\[
[\theta_1,\ldots,\theta_n]= \Cg^{\textbf{A}}(\{(0,f(b_1,\ldots,b_n)): f \in \Ab^{n}_{0} \textbf{A}, (0,b_i) \in \theta_i \}).
\]
\end{lemma}
\begin{proof}
Assume $\textbf{A}$ is nilpotent of degree $k$. Since $\textbf{A}$ is a nilpotent algebra in variety with a weak difference term, $\mathcal V(\textbf{A})$ is a Mal'cev variety. Let $m(x,y,z)$ be a Mal'cev term for $\textbf{A}$ and set $S= \{(0,f(b_1,\ldots,b_n)): f \in \Ab^{n}_{0} A, (0,b_i) \in \theta_i \}$. From \cite[Lem 6.9]{supernil}, we know that $[\theta_1,\ldots,\theta_n]$ is generated as a congruence by the set of pairs 
\[
\{(f(a_1,\ldots,a_n), f(b_1,\ldots,b_n)): f \text{ absorbs } (a_1,\ldots,a_n), (a_i,b_i) \in \theta_i \}.
\]
By inclusion of generating sets, $\Cg^{\textbf{A}}(S) \leq [\theta_1,\ldots,\theta_n]$. To reverse the inequality, take $a=f(a_1,\ldots,a_n), b=f(b_1,\ldots,b_n)$ such that $f$ absorbs $(a_1,\ldots,a_n)$ to $a$ and $(a_i,b_i) \in \theta_i$. From the discussion in Section \ref{sec:nilpotent}, there are unary polynomials $h_i(x):=f_k(x,0,a_i)$ for $i=1,\ldots,n$ and $h(x):= f_k(x,a,0)$ such that 
\begin{eqnarray*}
m(h_i(x),0,a_i)) &= \ x \ =& h_i(m(x,0,a_i)) \\
m(h(x),a,0) &= \ x \ =& h(m(x,a,0)).
\end{eqnarray*}  
Define the polynomial $t(x_1,\ldots,x_n):=m(f(m(x_1,0,a_1),\ldots,m(x_n,0,a_n)),a,0)$ and note $t \in \Ab^{n}_{0} \textbf{A}$. Then $a_i \, \theta_i \, b_i$ implies $h_i(b_i) \, \theta_i \, h_i(a_i)=f_k(m(0,0,a_i),0,a_i)=0$ which yields 
\begin{eqnarray*}
0 \, &\Cg^{\textbf{A}}(S)& \, t(h_1(b_1),\ldots,h_n(b_n)) \\
&=& m(f(m(h_1(b_1),0,a_1),\ldots,m(h_n(b_n),0,a_n)),a,0) \\
&=& m(f(b_1,\ldots,b_n),a,0)=m(b,a,0).  
\end{eqnarray*}
Then compatibility implies $b=h(m(b,a,0)) \, \Cg^{\textbf{A}}(S) \, h(0)=h(m(a,a,0))=a$.
\end{proof}

\begin{example}
In this example we observe the necessity in Definition \ref{def:ntype}(3) that all nested commutators are trivial. Define $\textbf{B}= \left\langle \mathds{Z}_{8}, +, f(x,y,z) \right\rangle$ where ``$+$'' is addition modulo 8 and $f(x,y,z)=2xyz$ is computed by multiplication modulo 8. $\Con \textbf{B}$ is the 4-element chain $0 < \beta < \alpha < 1$ where $(i,j) \in \beta \Leftrightarrow i-j \equiv 4 \mod 8$ and $(i,j) \in \alpha \Leftrightarrow i - j$ is even. Since $f$ is multilinear, the only way to generate non-trivial polynomials of arities at least 4 is by repeated compositions using $f$. It is not difficult to see that $g(x_1,x_2,x_3,x_4) \equiv 0 \mod 4$ whenever $g(x_1,x_2,x_3,x_4) \in \Ab^{4}_{0} \textbf{B}$ and $\Ab^{k}_{0} \textbf{B}=\{\hat{0}\}$ for $k \geq 6$; thus, $\textbf{B}$ is 6-supernilpotent. Using Lemma \ref{lem:cong}, we can calculate the remaining commutators. 

The evaluation $f(1,1,1)=2$ yields $(0,2) \in [1,1,1] \Rightarrow \alpha \leq [1,1,1]$. Since $\textbf{B}$ is nilpotent, we can then conclude $[1,1]=[1,1,1]=\alpha$, $[1,\alpha]=\beta$ and $[1,\beta]=0$. The evaluation $f(1,1,2)=4$ yields $(0,4) \in [1,1,\alpha]$ which implies $\beta \leq [1,1,\alpha]=[1,1,[1,1,1]] \leq [1,1,1,1,1] \leq [1,1,1,1]$. Then $g(x_1,x_2,x_3,x_4) \equiv 0 \mod 4$ whenever $g(x_1,x_2,x_3,x_4) \in \Ab^{4}_{0} \textbf{B}$ yields $0=[1/\beta,1/\beta,1/\beta,1/\beta]=([1,1,1,1] \vee \beta)/\beta = [1,1,1,1]/\beta \Rightarrow [1,1,1,1]=\beta$ using (HC6). It also follows that $g(x_1,x_2,x_3,x_4) \equiv 0 \mod 8$ whenever some $x_i \in \{0,2,4,6\}$ which implies $0=[1,1,1,\alpha]=[1,1,1,[1,1]]$ by Lemma \ref{lem:cong}. Finally, $0=[1,\beta]=\left[1,\left[1,1,1,1\right]\right]$. 
\end{example}

\subsection{Abelian and 3-Supernilpotent}

In this section, we note how to derive the affine structure of abelian algebras and the characterization of 3-supernilpotent Mal'cev algebras \cite[Thm.3.3]{mudrinski} from Theorem \ref{thm:nabelian} and general considerations of terms in loops. Let $\textbf{A}$ be an algebra with Mal'cev term $m(x,y,z)$ and fix an element $0 \in A$. Making the following definitions 
\begin{itemize}

	\item $x \cdot y:=m(x,0,y)$ \ \ $(x,y \in A)$,
	
	\item $(t \cdot s)(\bar{x}):=m(t(\bar{x}),0,s(\bar{x}))$ \ \ $(t,s \in \Ab^{n}_{0} \textbf{A})$,
	
	\item $(t \circ s)(x_1,\ldots,x_{\max\{n,m\}}):= t(s(x_1,\ldots,x_m),x_2,\ldots,x_n)$ \ \ $(t \in \Ab^{n}_{0} \textbf{A}, s \in \Ab^{m}_{0} \textbf{A})$,
	
	\item $t \ast a:= t(a)$ \ \ $(t \in \Ab^{1}_{0} \textbf{A}, a \in A)$,

\end{itemize}
we see that 
\begin{itemize}

	\item $0 \cdot x = x \cdot 0 = 0$ \ \ $(x \in A)$;
	
	\item $(t \cdot s) \circ r = (t \circ r) \cdot (s \circ r)$ \ \ $(t,s \in \Ab^{n}_{0} \textbf{A}, r \in \Ab^{m}_{0} \textbf{A})$;
	
	\item $(t \cdot s)\ast a = (t \ast a) \cdot (s \ast a)$ \ \ $(t,s \in \Ab^{1}_{0} \textbf{A}, a \in A)$.

\end{itemize}
Additional properties will depend on the behavior of the absorbing polynomials and the loop structure induced by supernilpotence.

Now fix a loop $\textbf{L} = \left\langle A, \cdot, \backslash, / , 0 \right\rangle$. It follows from the loop identities that for any $a \in A$ the maps
\begin{eqnarray*}
x &\mapsto& a \cdot x \ \ \ \ \ \ \  x \mapsto a \backslash x\\
x &\mapsto& x \cdot a \ \ \ \ \ \ \  x \mapsto a / x 
\end{eqnarray*}
are permutations on $A$. We can then define several terms which in some manner ``measure'' the failure of the operation $\cdot$ to be commutative or associative. Define the terms $[x,y], [x^{-1},y^{-1}]$ and $[y^{-1},x]$ as the unique solutions to the equations
\[
xy = [x,y] \cdot (yx) \ , \ \ \ \ xy = (yx) \cdot [x^{-1},y^{-1}] \ , \ \ \ \ xy = y \cdot ([y^{-1},x]  \cdot x ).
\]
Define the terms $a_1(x,y,z)$ and $a_2(x,y,z)$ as the unique solutions to the equations
\[
(xy)z = x ((yz) \cdot a_1(x,y,z))  \ \ \ \ \   \mathrm{and} \ \ \ \ \  (xy)z = (x(yz)) \cdot a_2(x,y,z).    
\]
For a polynomial $t \in \Ab^{n}_{0} \textbf{L}$, we can define $d_t(x,y,x_2,\ldots,x_n)$ as the solution to 
\[
t(x\cdot y,x_2,\ldots,x_n) = (t(x,x_2,\ldots,x_n) \cdot t(y,x_2,\ldots,x_n)) \cdot d_t(x,y,x_2,\ldots,x_n).
\]   
We can actually solve for the terms defined above; for example, 
\[
[y^{-1},x] = (y \backslash (xy))/ x  \ , \ \  a_1(x,y,z) = (yz) \backslash (x \backslash ((xy)z)).
\]
By direct calculation, it is easy to see the following is true.

\begin{lemma}\label{lem:assoc}
With the above definitions: 
\begin{enumerate}

	\item $[x,y], [x^{-1},y^{-1}] ,[y^{-1},x] \in \Ab^{2}_{0} \textbf{L}$, 
	
	\item $a_1(x,y,z), a_2(x,y,z) \in \Ab^{3}_{0} \textbf{L}$,
	
	\item $d_t(x,y,x_2,\ldots,x_n) \in \Ab^{n+1}_{0} \textbf{L}$.

\end{enumerate}
\end{lemma}

Several other terms can be defined in a similar manner, and all are easily seen to be absorbing polynomials and can claim to ``measure'' the lack of associativity, commutativity or distributivity of absorbing polynomials. If the algebra is supernilpotent, the derived operations ``$\cdot,\circ,\ast$'' induce additional algebraic structures on the absorbing polynomials.

Assume $\textbf{A}$ is a $n$-supernilpotent Mal'cev algebra. For $k \in \mathds{N}$, denote the algebra $\textbf{A}_k = \left\langle \Ab^{k}_{0} \textbf{A}, \cdot , \backslash , / , \hat{0} \right\rangle$ of $k$-ary absorbing polynomial with the loop operations induced by $\textbf{L}_A$. Since $\textbf{A}$ is $n$-supernilpotent, $\textbf{A}_k$ is trivial for $k \geq n$. Using the degree n-1 representation of $\textbf{A}$ in Theorem \ref{thm:nabelian}(2), we see that in general $\textbf{A}_k$ is n-(k-1)-supernilpotent for $1 \leq k < n$. Using Lemma \ref{lem:distr} applied to the derived absorbing loop polynomials in Lemma \ref{lem:assoc} we can say more; for $\left\lceil \frac{n}{3} \right\rceil \leq k < n$, $\textbf{A}_k$ is term equivalent to a group, and for $\left\lceil \frac{n}{2} \right\rceil \leq k < n$ it is abelian. For $t \in \Ab^{m}_{0} \textbf{A}$ define the maps $\phi_t, \psi_t : \textbf{A}_{k} \rightarrow \textbf{A}_{\max\{k,m\}}$ by 
\begin{eqnarray*}
\phi_t(f) &=& (t \circ f )(x_1,\ldots, x_{\max\{k,m\}})= t(f(x_1,\ldots,x_{k}),x_2,\ldots,x_m) \\
\psi_t(f) &=& (f \circ t )(x_1,\ldots, x_{\max\{k,m\}})= f(t(x_1,\ldots,x_{m}),x_2,\ldots,x_k). 
\end{eqnarray*}

By calculating the arities of absorbing polynomials composing with the terms in Lemma \ref{lem:assoc} and using Lemma \ref{lem:distr}, we have the following general representations:
\begin{enumerate}

	\item Assume $m + 2k -1 \geq n$ and $k \geq m$;
	
		\begin{enumerate}
		
			\item If $\left\lceil \frac{n}{3} \right\rceil \leq k < n$, then $\phi_t \in \End(\textbf{A}_k)$;
			
			\item If $\left\lceil \frac{n}{2} \right\rceil \leq k < n$, then $t \mapsto \phi_t: \left\langle \Ab^{m}_{0} \textbf{A}, \cdot, \backslash, /, \circ \right\rangle \rightarrow \End(\textbf{A}_k)$ is a right-nearring homomorphism;
		
		\end{enumerate}

	\item For all $1 \leq m,k < n$, $\psi_t \in \End(\textbf{A}_k)$;  
		
		\begin{enumerate}
		
			\item Assume $k \geq m$. If $\left\lceil \frac{n}{2} \right\rceil \leq k < n$ and $2m + k-1 \geq n$, then $t \mapsto \psi_t: \left\langle \Ab^{m}_{0} \textbf{A}, \cdot, \backslash, /, \circ \right\rangle \rightarrow \End(\textbf{A}_k)$ is a right-nearring homomorphism;
		
		\end{enumerate}		
	
	\item If $m > k$, then $\psi_t \in \Hom (\textbf{A}_k,\textbf{A}_m)$ 
	
		\begin{enumerate}
		
			\item If $\left\lceil \frac{n}{2} \right\rceil \leq m < n$, then for $1 \leq k < n$ the map $t \mapsto \psi_t: \textbf{A}_m \rightarrow \Hom (\textbf{A}_k,\textbf{A}_m)$ is a group homomorphism;
		
			\item If $\left\lceil \frac{n}{2} \right\rceil \leq m < n$ and $2k + m-1 \geq n$, then $t \mapsto \phi_t: \textbf{A}_m \rightarrow \Hom (\textbf{A}_k,\textbf{A}_m)$ is a group homomorphism.
		
		\end{enumerate}

\end{enumerate}

These actions and homomorphisms suggest the following meta-problem: to what extant does the sequence of derived loops $\textbf{L}_A,\textbf{A}_1,\ldots, \textbf{A}_{n-1}$ and their representations play a role in understanding the structure of the $n$-supernilpotent Mal'cev algebra $\textbf{A}$? We end this section with the two previous characterizations illustrating the calculations and representations in (1)-(3) above.

\begin{example}
Suppose $\textbf{A}$ is an abelian Mal'cev algebra. Then from Theorem \ref{thm:nabelian}, $\textbf{A}$ is polynomially equivalent to a 2-type loop $\textbf{L}_A = \left\langle A, \cdot, \backslash, / , 0, F \right\rangle$ which determines a degree 1 representation
\[
f(x_1,\ldots,x_n) = c \cdot r_1 = c \cdot \left( ( \cdots (t_1(x_1) \cdot t_2(x_2)) \cdots ) \cdot t_n(x_n) \right) 
\]
for $f \in \Pol_n \textbf{A}$. The operations of $\textbf{L}_A$ induce a loop structure $\textbf{A}_1$. Since $\textbf{A}$ is abelian, $\Ab^{k}_{0} \textbf{A}=\{\hat{0}\}$ for $k \geq 2$. Then associativity and commutativity of ``$\cdot$'' follows since the terms in Lemma \ref{lem:assoc}(1)-(2) are trivial; thus, the loop structures are term equivalent to abelian groups $\left\langle A, +, - , 0 \right\rangle$ and $\left\langle \Ab^{1}_{0} \textbf{A}, + , - , \hat{0} \right\rangle$ where $-x = x \backslash 0 = 0 / x$. Distributivity of ``$\circ$'' and ``$\ast$'' over ``$+$'' follows since $d_t(x,y)=\hat{0}$; therefore, $\textbf{R}_A = \left\langle \Ab^{1}_{0} \textbf{A}, +, -, 0, \circ \right\rangle$ is a ring and ``$\ast$'' makes $\left\langle A, + , -, 0 \right\rangle$ into a $\textbf{R}_A$-module where the degree 1 representation explicitly describes the module polynomials.
\end{example}

\begin{example}\label{ex:3abel}
Suppose $\textbf{A}$ is a 3-supernilotent Mal'cev algebra. Then $\Ab^{k}_{0} \textbf{A} = \{\hat{0}\}$ for $k \geq 3$. From Theorem \ref{thm:nabelian}, $\textbf{A}$ is polynomially equivalent to a 3-type loop $\textbf{L}_A = \left\langle A, \cdot, \backslash, / , 0, F \right\rangle$ where $F = \Ab^{1}_{0} \textbf{A} \cup \Ab^{2}_{0} \textbf{A}$.  The operations of $\textbf{L}_A$ define loops $\textbf{A}_1$ and $\textbf{A}_2$. We no longer have commutativity, but associativity of ``$\cdot$'' follows from $a_1(x,y,z)=\hat{0}$; therefore, the loops $\left\langle A, \cdot, \backslash, / , 0, \right\rangle$ and $\textbf{A}_1$ are term equivalent to the groups $\left\langle A, \cdot, ^{-1}, 0 \right\rangle$ and  $\left\langle \Ab^{1}_{0} \textbf{A}, \cdot , ^{-1} , \hat{0} \right\rangle$. Then the expanded group $\left\langle A, \cdot, ^{-1}, 0, F \right\rangle$ is of the type in \cite[Thm 3.3]{mudrinski} by Theorem \ref{thm:nabelian}(3). 

There is additional structure. Since we always have right-distributivity of ``$\circ$'' over ``$\cdot$'', the structure $\left\langle \Ab^{1}_{0} \textbf{A}, \cdot , ^{-1} , \hat{0}, \circ \right\rangle$ is a right-nearring. For $f,g \in \Ab^{2}_{0} \textbf{A}$, $([f(x,y),g(x,y)], 0) \in [[1,1],[1,1]]=0_A$ for all $x,y \in A$ which implies $[f,g]=\hat{0}$; thus, together with $a_1(x,y,z)=\hat{0}$ we see that the loop $\textbf{A}_2$ is term equivalent to the abelian group $\left\langle \Ab^{2}_{0} \textbf{A}, + , - , \hat{0} \right\rangle$. Note the composition ``$\circ$'' is trivial in $\Ab^{2}_{0} \textbf{A}$.  

For $f,g \in \Ab^{2}_{0} \textbf{A}$ and $t \in \Ab^{1}_{0} \textbf{A}$, we see that $d_t(f,g)=\hat{0}$ in Lemma \ref{lem:assoc}(3) which implies that $\phi_t(f)$ determines an action via $\circ$ of the right-nearring $\left\langle \Ab^{1}_{0} \textbf{A}, \cdot , ^{-1} , \hat{0}, \circ \right\rangle$ on the abelian group $\left\langle \Ab^{2}_{0} \textbf{A}, + , - , \hat{0} \right\rangle$. For fixed $t \in \Ab^{1}_{0} \textbf{A}$, the action implies the map $\Ab^{2}_{0} \textbf{A} \ni f \longmapsto \phi_t(f) \in \Ab^{2}_{0} \textbf{A}$ is an endomorphism of the abelian group; therefore, $\phi_t \in \End \left\langle \Ab^{2}_{0} \textbf{A}, + , - , \hat{0} \right\rangle$ which is naturally a right-nearring under addition and composition of endomorphisms. Then the map $t \longmapsto \phi_t: \left\langle \Ab^{1}_{0} \textbf{A}, \cdot , ^{-1} , \hat{0}, \circ \right\rangle \rightarrow \End \left\langle \Ab^{2}_{0} \textbf{A}, + , - , \hat{0} \right\rangle$ is a right-nearring homomorphism. Similarly, pre-composition 
\[
t \longmapsto \psi_t: \left\langle \Ab^{1}_{0} \textbf{A}, \cdot , ^{-1} , \hat{0}, \circ \right\rangle \rightarrow \End \left\langle \Ab^{2}_{0} \textbf{A}, + , - , \hat{0} \right\rangle
\]
is also a right-nearring homomorphism.   
\end{example}

\section{Neutrality and Weak Difference Terms}\label{sec:neutral}

The $n$-commutator is \textit{neutral} in an algebra $\textbf{A}$ if $[\alpha_1,\ldots,\alpha_n] = \alpha_1 \wedge \cdots \wedge \alpha_n$ for all $\{\alpha_1,\ldots,\alpha_n\} \subseteq \Con \textbf{A}$. For a variety $\mathcal V$, the $n$-commutator is said to be neutral if it is neutral in every algebra in the variety. The next lemma is a standard result for the binary commutator in arbitrary varieties and can be argued in exactly the same manner for the higher commutators.

\begin{lemma}\label{lem:abel}
Let $\mathcal V$ be a variety. The following are equivalent for $n \geq 2$:
\begin{enumerate}

	\item $[\theta(x,y),\ldots,\theta(x,y)]_{n} < \theta(x,y)$ in $\Con \textbf{F}_{\mathcal V}(2)$;
	
	\item $\mathcal V$ contains an algebra with a proper $n$-supernilpotent interval in its congruence lattice.   

\end{enumerate}
\end{lemma}

\begin{lemma}\label{lem:abelinterv}
Suppose $\textbf{A}$ is a Mal'cev algebra and $(\alpha,\beta)$ is an $n$-supernilpotent interval in $\Con \textbf{A}$. Then for every cover $\alpha \leq \gamma \prec \sigma \leq \beta$, $(\gamma, \sigma)$ is an abelian interval. 
\end{lemma}
\begin{proof}
By assumption, $\textbf{A} \vDash C(\beta,\ldots,\beta;\alpha)$ and so $[\sigma,\ldots,\sigma]_n \leq [\beta,\ldots,\beta]_n \leq \alpha \leq \gamma$ by (HC2). Then by (HC5) we have $\textbf{A} \vDash C(\sigma,\ldots,\sigma;\gamma)$ and so $\textbf{A}/\gamma \vDash C(\sigma/\gamma,\ldots,\sigma/\gamma;0_{A/\gamma})$. Note $0_{A/\gamma} \prec \sigma/\gamma$. Then using (HC8) on the iterated binary commutators
\begin{eqnarray*}
[\sigma/\gamma,[\sigma/\gamma,\cdots,[\sigma/\gamma,\sigma/\gamma]\cdots ]] \leq [\sigma/\gamma,\ldots,\sigma/\gamma]_{n} &=& 0_{A/\gamma} \\
&\leq& [\sigma/\gamma,\ldots,\sigma/\gamma]_{k} \\
&\leq& \sigma/\gamma
\end{eqnarray*} 
for $2 \leq k < n$ implies $[\sigma/\gamma,\sigma/\gamma]=0_{A/\gamma}$. This yields $\textbf{A}/\gamma \vDash C(\sigma/\gamma,\sigma/\gamma;0_{A/\gamma})$ which implies $\textbf{A} \vDash C(\sigma,\sigma;\gamma)$ by Lemma \ref{lem:cen}(4).
\end{proof}

If we no longer assume $\textbf{A}$ has a Mal'cev term but is a finite algebra, then we can still conclude every cover in a supernilpotent interval is abelian.

\begin{proposition}\label{prop:tct}
Suppose $\textbf{A}$ is a finite algebra and $(\alpha,\beta)$ a supernilpotent interval in $\Con \textbf{A}$. Then for every cover $\alpha \leq \gamma \prec \sigma \leq \beta$, $(\gamma, \sigma)$ is an abelian interval.
\end{proposition}
\begin{proof}
Suppose $(\alpha,\beta)$ is a $n$-supernilpotent interval in $\Con \textbf{A}$ and $\alpha \leq \gamma \prec \sigma \leq \beta$. Then $\textbf{A} \vDash C(\underbrace{\sigma,\ldots,\sigma}_{n};\alpha)$ since $\sigma \leq \beta$. In the language of tame congruence theory \cite{TCT}, we shall show type$(\gamma,\sigma) \in \{1,2\}$.

For a contradiction, assume type$(\gamma,\sigma) \in \{3,4,5\}$. Fix a minimal set $U \in M(\gamma,\sigma)$ and its unique trace $N$. If type$(\gamma,\sigma) \in \{3,4\}$, then for the induced algebra we have $\textbf{A}_N = \left\langle \{0,1\}, (\Pol \textbf{A})_N \right\rangle$ with $(0,1) \in \sigma - \gamma$ and $\textbf{A}_N$ polynomially equivalent to either a Boolean algebra (type 3) or a lattice (type 4). If type$(\gamma,\sigma)=5$, then $N/\gamma = \{0/\gamma,1/\gamma\}$ for two elements $\{0,1\} \subseteq A$ and $\textbf{A}/\gamma_{N/\gamma}$ is polynomially equivalent to a semilattice with $1/\gamma=\{1\}$ and minimal element $\perp=0/\gamma$; in addition, for any $a \in 0/\gamma$ there is a polynomial $t$ such that $\left\langle \{a,1\},t\restriction_{\{a,1\}} \right\rangle$ is a semilattice with minimal element $\perp=a$.

In all cases, we can find a polynomial $t$ of $\textbf{A}$ and two elements $\{0,1\}$ such that $(0,1) \in \sigma - \gamma$ and $t$ restricted to $\{0,1\}$ is a semilattice operation with minimal element $\perp=0$. Define the polynomial $h(x_1,\ldots,x_n)=t(x_1,t(x_2,\cdots, t(x_{n-1},x_n) \cdots))$. Using the semilattice identities, for all $\bar{u} \in \prod_{i=1}^{n-1} \{0,1\} \, \backslash \, \{1,\ldots,1\}$ we see that $h(\bar{u},0) = 0 = h(\bar{u},1)$ because some coordinate in $\bar{u}$ is $0$. Since $h$ is still a polynomial of $\textbf{A}$ we conclude that $0=h(1,\ldots,1,0) \, [\sigma,\ldots,\sigma]_n \, h(1,\ldots,1,1)=1$; thus, $(0,1) \in [\sigma,\ldots,\sigma]_{n} \leq \alpha \leq \gamma$, a contradiction. It must be that type$(\gamma,\sigma) \in \{1,2\}$.
\end{proof}

\begin{corollary}\label{cor:supmal}
Let $\textbf{A}$ be a finite algebra. Every supernilpotent congruence is solvable; consequently, if $\textbf{A}$ is a supernilpotent Taylor algebra, then $\textbf{A}$ is a nilpotent Mal'cev algebra.
\end{corollary}
\begin{proof}
Assume $\textbf{A}$ is a finite algebra and $\alpha \in \Con \textbf{A}$ is a supernilpotent congruence for some $n \geq 2$. Then by Proposition \ref{prop:tct}, $\textbf{A} \vDash C(\alpha,\alpha;\sigma)$ for all co-atoms $\sigma \prec \alpha$; thus, $[\alpha,\alpha] \leq \bigwedge_{\sigma \prec 1} \sigma < \alpha$. Since $(0,[\alpha,\alpha])$ is also a supernilpotent interval, Proposition \ref{prop:tct} again shows $[[\alpha,\alpha],[\alpha,\alpha]] \leq \bigwedge_{0 \leq \sigma \prec [\alpha,\alpha]} \sigma < [\alpha,\alpha]$ provided $[\alpha,\alpha] \neq 0$. Inductively, we have a descending derived series $\alpha > [\alpha,\alpha] > [\alpha]^2 > \cdots $. Since $\textbf{A}$ is finite, some $[\alpha]^k=0$ which shows $\alpha$ is solvable. 

If $\textbf{A}$ is a Taylor algebra and supernilpotent, then the first part applied to $1_A$ shows $\textbf{A}$ is solvable. According to \cite[Thm.9.6]{TCT}, $\textbf{A}$ interprets a Mal'cev term; thus, $\textbf{A}$ is nilpotent because it is a supernilpotent Mal'cev algebra. 
\end{proof}

\begin{theorem}\label{thm:neutral}
Let $\mathcal V$ be a variety in which the 2-generated free algebra $\textbf{F}_{\mathcal V}(2)$ is finite. The following are equivalent:
\begin{enumerate}
	
	\item $\mathcal V$ is congruence meet-semidistributive;
	
	\item $\mathcal V$ is $n$-commutator neutral for all $n \geq 2$;
	
	\item $\mathcal V$ is $n$-commutator neutral for some $n \geq 2$.

\end{enumerate}
\end{theorem}
\begin{proof}
We show $(1) \Rightarrow (2)$. Suppose $\mathcal V$ is congruence meet-semidistributive but there exists $n > 2$ and $\textbf{A} \in \mathcal V$ with $\alpha_1,\ldots,\alpha_n \in \Con \textbf{A}$ such that $[\alpha_1,\ldots,\alpha_n] < \alpha_1 \wedge \cdots \wedge \alpha_n$. Set $\beta = \alpha_1 \wedge \cdots \wedge \alpha_n$ and note $[\beta,\ldots,\beta]_n < \beta$; thus, $([\beta,\ldots,\beta]_n,\beta)$ is an $n$-supernilpotent interval in $\Con \textbf{A}$. By Lemma \ref{lem:abel}, $([\theta,\ldots,\theta]_n,\theta)$ is a nontrivial $n$-supernilpotent interval in $\Con \textbf{F}_{\mathcal V}(2)$ where $\theta = \theta(x,y)$. Since $\textbf{F}_{\mathcal V}(2)$ is finite, we can find an n-supernilpotent cover $[\theta,\ldots,\theta]_n \prec \gamma \leq \theta$. By Proposition \ref{prop:tct},
$[\theta,\ldots,\theta]_n  \prec \gamma$ is an abelian cover, but this cannot occur in a congruence meet-semidistributive variety \cite[Cor.4.7]{twocommutator}; therefore, every $n$-commutator is neutral in $\mathcal V$. 

The implication $(2) \Rightarrow (3)$ is trivial.

We now show $(3) \Rightarrow (1)$. Assume the $n$-commutator is neutral for some $n \geq 2$. If $n > 2$, then for any $\alpha,\beta \in \Con \textbf{A}$, $\textbf{A} \in \mathcal V$,
\[
\alpha \wedge \beta = [\underbrace{\alpha,\ldots,\alpha}_{n-1},\beta] \leq [\alpha,\beta] \leq \alpha \wedge \beta
\]
shows the binary commutator is neutral and so $\mathcal V$ must be congruence meet-semidistributive \cite[Cor.4.7]{twocommutator}. 
\end{proof}

There are congruence meet-semidistributive varieties which are not locally finite but the two-generated free algebra is finite; for example, the variety of 2-semilattices. The implication $(2) \Rightarrow (1)$ is always true so the non-trivial direction is to establish $(1) \Rightarrow (2)$. We conjecture the characterization holds without additional assumptions of finiteness in the variety.

\begin{conjecture}
A variety is congruence meet-semidistributive if and only if all the higher commutators are neutral.
\end{conjecture}

For locally finite varieties, tame congruence theory characterizes Taylor varieties by the existence of an idempotent ternary term which interprets as a Mal'cev operation on the blocks of any locally solvable congruence \cite[Thm.9.6]{TCT}; consequently, using \cite{herrmann} we can see that a locally finite variety is Taylor if and only if the abelian algebras are affine. Without any assumption of finiteness, the class of varieties for which the abelian algebras are affine form a proper subclass of Taylor varieties. A \textit{weak difference} term for a variety $\mathcal V$ is a term $c(x,y,z)$ such that for all $a,b \in \textbf{A} \in \mathcal V$ we have
\[
c(a,b,b) \, [\theta,\theta] \, a \, [\theta,\theta] \, c(b,b,a)
\] 
where $\theta=\theta(a,b) \in \Con \textbf{A}$. For $\alpha,\beta,\gamma \in \Con \textbf{A}$, recursively define congruences in the following manner: $\beta_0=\beta$, $\gamma_0=\gamma$ and $\gamma_{k+1} = \gamma \vee (\alpha \wedge \beta_k)$, $\beta_{k+1} = \beta \vee (\alpha \wedge \gamma_k)$.

\begin{theorem}\cite[Thm.4.8]{twocommutator}\label{thm:weakdiff}
Let $\mathcal V$ be a variety. The following conditions are equivalent:
\begin{enumerate}

	\item $\mathcal V$ has a weak difference term;
	
	\item $\mathcal V$ satisfies the congruence inclusion $\alpha \wedge ( \beta \circ \gamma) \subseteq \gamma_n \circ \beta_n$ for some $n \in \mathds{N}$.  
	
	\item $\mathcal V$ satisfies an idempotent Mal'cev condition which is strong enough to imply that
abelian algebras are affine.

\end{enumerate}
\end{theorem}

An explicit presentation of the Mal'cev condition for varieties with a weak difference term can be derived in the standard way from the congruence inclusion stated in Theorem \ref{thm:weakdiff}(2).

In light of Theorem \ref{thm:nabelian}, it is reasonable to conjecture that the class of varieties in which the $n$-supernilpotent algebras are polynomially equivalent to n-type loops determines a Mal'cev condition - it is easy to see that such varieties have a weak difference term. We shall show that for locally finite varieties the two classes are equivalent. A \textit{weak n-difference} term for a variety $\mathcal V$ is a term $c(x,y,z)$ such that for all $a,b \in \textbf{A} \in \mathcal V$ we have
\[
c(a,b,b) \, [\theta,\ldots,\theta]_{n} \, a \, [\theta,\ldots,\theta]_{n} \, c(b,b,a).
\] 
where $\theta=\theta(a,b) \in \Con \textbf{A}$.

\begin{theorem}\label{thm:ndifference}
Let $\mathcal V$ be a variety such that $\textbf{F}_{\mathcal V}(2)$ is finite. The following conditions are equivalent:
\begin{enumerate}

	\item $\mathcal V$ has a term $c(x,y,z)$ which is a weak n-difference term for all $n \geq 2$;
	
	\item $\mathcal V$ satisfies an idempotent Mal'cev condition which is strong enough to imply that for all $n \geq 2$, n-supernilpotent algebras are polynomially equivalent to n-type loops;
	
	\item $\mathcal V$ has a weak difference term.

\end{enumerate}
\end{theorem}
\begin{proof}
Consider the statement

$(2)'$: $\mathcal V$ has an idempotent term which interprets as a Mal'cev operation in every block of a supernilpotent congruence.
\vspace{0.1cm}

From Theorem \ref{thm:nabelian}, it is enough to show $(1)$ and $(3)$ are equivalent to $(2)'$. The Mal'cev condition in (2) is then the one given by Theorem \ref{thm:weakdiff}(2).

$(1) \Rightarrow (2)'$: A weak n-difference term must be idempotent because $\theta(a,a)=0_A$ for all $a \in A$ $\textbf{A} \in \mathcal V$. Since the n-commutator $[\theta,\ldots,\theta]=0_A$ in a n-supernilpotent algebra $\textbf{A}$, the weak n-difference term satisfies the Mal'cev identities.

$(2)' \Rightarrow (3)$: Let $c(x,y,z)$ be an idempotent term which interprets as a Mal'cev operation in every block of a supernilpotent congruence. Let $\theta$ be a congruence of $\textbf{A} \in \mathcal V$ and $(a,b) \in \theta$. If $\theta = [\theta, \theta]$, then 
\[
c(b,b,a) \, [\theta, \theta] \, a \, [\theta, \theta] \, c(a,b,b)
\] 
because $c$ is idempotent. In case $[\theta, \theta] < \theta$, we factor by $[\theta, \theta]$ and observe that $\theta/[\theta, \theta]$ is an abelian congruence in $\textbf{A}/[\theta,\theta]$. Then $c(a,b,b)/[\theta, \theta] = a/[\theta, \theta] = c(b,b,a)/[\theta, \theta]$. It follows that $c(x,y,z)$ is a weak difference for $\mathcal V$.

$(3) \Rightarrow (1)$: If $\mathcal V$ has no supernilpotent congruences, then it is congruence meet-semidistributive by Theorem \ref{thm:neutral} and so the n-commutators are all neutral. The third projection term will then serve as the required n-difference term. We may assume $\mathcal V$ contains nontrivial n-supernilpotent congruences. If we take $n \geq 2$, then by Lemma \ref{lem:abel} $([\theta,\ldots,\theta]_n, \theta)$ is a proper n-supernilpotent interval where $\theta=\theta(x,y) \in \Con \textbf{F}_{\mathcal V}(2)$. Since $\textbf{F}_{\mathcal V}(2)$ is finite, we may choose $n \in \mathds{N}$ large enough such that $[\theta,\ldots,\theta]_k = [\theta,\ldots,\theta]_n < \theta$ for $k \geq n$.

The free algebra $\textbf{F}_{\mathcal V}(2)$ is Taylor because it is finite and contained in a variety with a weak difference term; therefore, there is an idempotent term $c(x,y,z)$ in the subvariety $\mathcal V(\textbf{F}_{\mathcal V}(2)) \leq \mathcal V$ which interprets as a Mal'cev operation in every block of a solvable congruence in the subvariety \cite[Thm.9.6]{TCT}. Since $\theta/[\theta,\ldots,\theta]_n$ is a supernilpotent congruence in $\textbf{F}_{\mathcal V}(2)/[\theta,\ldots,\theta]_n$, by Corollary \ref{cor:supmal} it is solvable. Then $c(x,y,y)/[\theta,\ldots,\theta]_n = x/[\theta,\ldots,\theta]_n = c(y,y,x)/[\theta,\ldots,\theta]_n$ because $(x,y) \not\in [\theta,\ldots,\theta]_n$. This shows $c$ is a weak n-difference term for $\mathcal V$. By the choice of $n$, it is a weak m-difference term in $\mathcal V$ for all $m \geq 2$.
\end{proof}

\begin{corollary}
Let $\mathcal V$ be a locally finite Taylor variety. An algebra in $\mathcal V$ is n-supenilpotent if and only if it polynomially equivalent to a n-type loop.
\end{corollary}

We conjecture that Theorem \ref{thm:ndifference} continues to be true without the finiteness assumption with one important proviso - we may no longer expect to find a fixed ternary term which is a weak n-difference term for all $n \geq 2$. The truth of the following conjecture would show that the class of varieties in which the n-supernilpotent algebras are polynomially equivalent to n-type loops is a Mal'cev condition and determines the class of varieties with a weak difference term.

\begin{conjecture}\label{conj:weakcomm}
For a variety $\mathcal V$ the following are equivalent:
\begin{enumerate}

	\item for each $n \geq 2$ there is a term $c_n(x,y,z)$ which is an weak n-difference term for $\mathcal V$;
	
	\item $\mathcal V$ satisfies an idempotent Mal'cev condition which is strong enough to imply that for all $n \geq 2$, n-supernilpotent algebras are polynomially equivalent to n-type loops;
	
	\item $\mathcal V$ has a weak difference term.

\end{enumerate}
\end{conjecture}

We shall show the implication $(3) \Rightarrow (1)$ in Conjecture \ref{conj:weakcomm} holds for the important subclass of congruence modular varieties. A \textit{n-difference} term for a variety $\mathcal V$ is a term $c(x,y,z)$ such that for all $a,b \in \textbf{A} \in \mathcal V$ we have
\[
c(a,b,b) \, [\theta,\ldots,\theta]_{n} \, a \ \ \ \text{and} \ \ \  c(b,b,a)=a.
\] 
where $\theta=\theta(a,b) \in \Con \textbf{A}$.

\begin{theorem}\label{thm:cmdiff}
If $\mathcal V$ is a congruence modular variety, then $\mathcal V$ has a n-difference term $c_n(x,y,z)$ for each $n \geq 2$.
\end{theorem}
\begin{proof}
From Gumm \cite{gummterms}, $\mathcal V$ is congruence modular if and only if there are terms $d_1,\ldots, d_n, q$ such that the following identities
\begin{eqnarray*}
x &=& d_1(x,y,z) \\
x &=& d_i(x,y,x) \ \ \ \ \ \ \text{ for all } i \\
d_i(x,x,z) &=& d_{i+1}(x,x,z) \ \ \ \text{ for even } i \\
d_i(x,z,z) &=& d_{i+1}(x,z,z) \ \ \ \text{ for odd } i \\
d_n(x,z,z) &=& q(x,z,z) \\
q(x,x,z) &=& z
\end{eqnarray*}
hold throughout $\mathcal V$. Recursively define terms $q_n(x,y,z)$ for $n \geq 2$ by setting $q_2(x,y,z):=q(x,y,z)$ and $q_{m+1}(x,y,z):=q(x,q_m(x,y,y),q_m(x,y,z))$. The claim is that $q_m$ is a m-difference term for $\mathcal V$. It is easy to see that $q_m(x,x,y)=y$ for all $m \geq 2$. The goal of the proof is to show $x \, [\theta,\ldots,\theta]_m \, q_m(x,y,y)$ when $(x,y) \in \theta$. This is well-known for $m=2$ (a nice development is in \cite[Thm 5.5,Thm 6.4]{commod} which synthesizes contributions from several authors). For $m \geq 3$, the strategy will be to apply the higher centralizer relation condition to a new recursively defined sequence of polynomials. Fix $\textbf{A} \in \mathcal V$ and $a,b \in A$.

Define the integer sequence $\{k_n\}$ by $k_3=3$ and $k_m=2k_{m-1}+1$. For $(v_1,v_2,v_3) \in [n]^{k_3}$ define
\[
t_{(v_1,v_2,v_3)}(x_1,x_2,x_3) := d_{v_1}(a, d_{v_2}(a,x_3,x_1),d_{v_3}(a,b,x_2)),
\] 
and then for $v \in [n]^{k_{m+1}}$ where $v=(i,w,u)$ with $w,u \in [n]^{k_m}$ define 
\begin{eqnarray*}
t_{v}(x_1,\ldots,x_{m+1}) &=& t_{(i,w,u)}(x_1,\ldots,x_{m+1}) \\
&:=& d_{i}(a, t_{w}(x_1,\ldots,x_{m-1},x_{m+1}),t_{u}(b,\ldots,b,x_{m},b)).
\end{eqnarray*}
Using the identity $d_n(x,z,z) = q(x,z,z)$, it is not difficult to see that for $(n,\ldots,n) \in [n]^{k_m}$, $t_{(n,\ldots,n)}(b,\ldots,b)=q_m(a,b,b)$. The task is to show $a \, [\theta,\ldots,\theta]_m \, t_{(n,\ldots,n)}(b,\ldots,b)$ where $(a,b) \in \theta \in \Con \textbf{A}$. The first step is to establish the following claim:

\begin{claim}
\begin{equation}\label{eq:induct3}
t_{v}(b,\ldots,b,a) \, [\theta,\ldots,\theta]_m \, t_{v}(b,\ldots,b,b) \ \ \ \text{for all} \ \ v \in [n]^{k_m}. 
\end{equation} 
\end{claim}
\begin{proof}
To do this, we shall argue inductively on $m \geq 3$ that
\begin{equation}\label{eq:induct1}
t_v(z_1,\ldots,z_{m-2},a,a)=a=t_v(z_1,\ldots,z_{m-2},a,b) 
\end{equation} 
\begin{equation}\label{eq:induct2}
t_v(w_1,\ldots,w_{m-2},b,a)=t_v(w_1,\ldots,w_{m-2},b,b) 
\end{equation}
for every $(z_1,\ldots,z_{m-2}) \in \{a,b\}^{m-2}$ and $(w_1,\ldots,w_{m-2}) \in \{a,b\}^{m-2} \backslash \{(b,\ldots,b)\}$. Since by definition $\textbf{A} \vDash \mathrm{C}(\theta,\ldots,\theta;[\theta,\ldots,\theta]_m)$, the evaluation (\ref{eq:induct3}) will then follow from (\ref{eq:induct1}) and (\ref{eq:induct2}) applied to the premise in the centralizer relation. For the base case $m=3$, if we let $v=(i,j,k)$, then using the identities $x=d_r(x,y,x)$ we have the evaluations
\begin{eqnarray*}
t_{v}(a,a,a) &=& d_{i}(a, d_{j}(a,a,a),d_{k}(a,b,a)) = d_i(a,a,a) = a \\
t_{v}(a,a,b) &=& d_{i}(a, d_{j}(a,b,a),d_{k}(a,b,a)) = d_i(a,a,a) = a \\
t_{v}(b,a,a) &=& d_{i}(a, d_{j}(a,a,b),d_{k}(a,b,a)) = d_i(a,d_j(a,a,b),a) = a \\
t_{v}(b,a,b) &=& d_{i}(a, d_{j}(a,b,b),d_{k}(a,b,a)) = d_i(a, d_j(a,b,b),a) = a 
\end{eqnarray*}
verifying (\ref{eq:induct1}) and then 
\begin{eqnarray*}
t_{v}(a,b,a) &=& d_{i}(a, d_{j}(a,a,a),d_{k}(a,b,b))=d_i(a,a,d_k(a,b,b)) \\ 
t_{v}(a,b,b) &=& d_{i}(a, d_{j}(a,b,a),d_{k}(a,b,b)) = d_i(a,a,d_k(a,b,b))
\end{eqnarray*}
to show (\ref{eq:induct2}). Now assume the result for $m \geq 3$. Take $v \in [n]^{k_{m+1}}$ where $v=(i,w,u)$ and $(z_1,\ldots,z_{m-1}) \in \{a,b\}^{m-1}$. The inductive hypothesis on the right-side of (\ref{eq:induct1}) yields $a=t_{u}(b,\ldots,b,a,b)$. Then  
\begin{eqnarray*}
t_{(i,w,u)}(z_1,\ldots,z_{m-1},a,a) &=& d_i(a,t_{w}(z_1,\ldots,z_{m-1},a),t_{u}(b,\ldots,b,a,b)) \\
 &=& d_i(a,t_{w}(z_1,\ldots,z_{m-1},a),a)=a \\
t_{(i,w,u)}(z_1,\ldots,z_{m-1},a,b) &=& d_i(a,t_{w}(z_1,\ldots,z_{m-1},b),t_{u}(b,\ldots,b,a,b)) \\
&=& d_i(a,t_{w}(z_1,\ldots,z_{m-1},b),a)=a \\
\end{eqnarray*}
yields (\ref{eq:induct1}) at the m+1-step. Since $(w_1,\ldots,w_{m-1}) \in \{a,b\}^{m-1} \backslash \{(b,\ldots,b)\}$ accounts for all possible tuples in the inductive hypothesis for step $m$, we have $t_{w}(w_1,\ldots,w_{m-1},a) = t_{w}(w_1,\ldots,w_{m-1},b)$ by (\ref{eq:induct1}) and (\ref{eq:induct2}). Then 
\begin{eqnarray*}
t_{v}(w_1,\ldots,w_{m-1},b,a) &=& d_i(a,t_{w}(w_1,\ldots,w_{m-1},a),t_{u}(b,\ldots,b,b)) \\
&=& d_i(a,t_{w}(w_1,\ldots,w_{m-1},b),t_{u}(b,\ldots,b,b)) \\
&=& t_{v}(w_1,\ldots,w_{m-1},b,b) 
\end{eqnarray*}
which verifies (\ref{eq:induct2}) at m+1-step. The inductive step is now complete and establishes the claim.
\end{proof}

To finish the theorem, we must show $a \, [\theta,\ldots,\theta]_m \, t_{(n,\ldots,n)}(b,\ldots,b)$ whenever $(n,\ldots,n) \in [n]^{k_m}$. This will be accomplished by using (\ref{eq:induct3}) together with the Gumm term identities. By using projections to extend the number of Gumm terms, we may assume $n$ is odd. The argument will calculate with polynomials $t_{v}(x_1,\ldots,x_m)$ for a particular subset of sequences $v \in [n]^{k_m}$ inductively defined in the following way: let $S_3=\{(i,j,n) \in [n]^{3}: i,j \in [n]\}$ and $S_m=\{ (i,u,n,\ldots,n) \in [n]^{k_m}: u \in S_{m-1}, i \in [n]\}$ for $m > 3$.

\begin{claim}\label{claim:reduce}
\begin{enumerate}

	\item Let $1 \leq r \leq m-1$ and $i_r$ even. If $(i_{1},\ldots,i_{r},n,\ldots,n) \in S_m$, then
	\[
t_{(i_1,\ldots,i_{r},n,\ldots,n)}(b,\ldots,b,b) = t_{(i_1,\ldots,i_{r}-1,n,\ldots,n)}(b,\ldots,b,b).
\]
	
	\item Let $i_{m-1}$ odd. If $(i_{1},\ldots,i_{m-1},n,\ldots,n) \in S_m$, then
\[
t_{(i_1,\ldots,i_{m-1},n,\ldots,n)}(b,\ldots,b,a) = t_{(i_1,\ldots,i_{m-1}-1,n,\ldots,n)}(b,\ldots,b,a). 
\]
	
	\item Let $1 \leq r \leq m-1$ and $i_r$ odd. If $(i_1,\ldots,i_{r},1,\ldots,1,n,\ldots,n) \in S_m$, then 
\[
t_{(i_1,\ldots,i_{r},1,\ldots,1,n,\ldots,n)}(b,\ldots,b,b) = t_{(i_1,\ldots,i_{r}-1,1,\ldots,1,n,\ldots,n)}(b,\ldots,b,b)
\]

	\item Let $1 \leq r \leq m-1$. If $(i_1,\ldots,i_r,1,j_1,\ldots,j_s,n,\ldots,n) \in S_m$, then for $p_1,\ldots,p_s \in [n]$,
	\[
	t_{(i_1,\ldots,i_r,1,j_1,\ldots,j_s,n,\ldots,n)}(b,\ldots,b) = t_{(i_1,\ldots,i_r,1,p_1,\ldots,p_s,n,\ldots,n)}(b,\ldots,b).
	\]

\end{enumerate}
\end{claim}
\begin{proof}
(1) Precisely because $(i_{1},\ldots,i_{r},n,\ldots,n) \in S_m$, we can expand the nested definition from left-to-right and use the identity $d_{i_r-1}(x,y,y)=d_{i_r}(x,y,y)$ to see
\[
t_{(i_1,\ldots,i_{r},n,\ldots,n)}(b,\ldots,b,b) \hspace{12cm}
\]
\begin{eqnarray*}
&=& d_{i_1}(a,d_{i_2}(a,\cdots d_{i_r}(a, t_{(n,\ldots,n)}(b,\ldots,b),t_{(n,\ldots,n)}(b,\ldots,b)) \cdots ))\\
&=& d_{i_1}(a,d_{i_2}(a,\cdots d_{i_r-1}(a, t_{(n,\ldots,n)}(b,\ldots,b),t_{(n,\ldots,n)}(b,\ldots,b)) \cdots )) \\
&=& t_{(i_1,\ldots,i_{r},n,\ldots,n)}(b,\ldots,b,b).
\end{eqnarray*}
(2) Similarly, using the identity $d_{i_{m-1}-1}(x,x,y)=d_{i_m}(x,x,y)$ we have
\begin{eqnarray*}
t_{(i_1,\ldots,i_{m-1},n,\ldots,n)}(b,\ldots,b,a)  &=& d_{i_1}(a,d_{i_2}(a,\cdots d_{i_{m-1}}(a,a,d_n(a,b,b)) \cdots ))\\
&=& d_{i_1}(a,d_{i_2}(a,\cdots d_{i_{m-1}-1}(a,a,d_n(a,b,b)) \cdots )) \\
&=& t_{(i_1,\ldots,i_{m-1}-1,n,\ldots,n)}(b,\ldots,b,a).
\end{eqnarray*}
(3) Since $(i_1,\ldots,i_{r},1,\ldots,1,n,\ldots,n) \in S_m$ and $x=d_1(x,y,z)$, when we expand the nested definition we have
\[
t_{(i_1,\ldots,i_{r},1,\ldots,1,n,\ldots,n)}(b,\ldots,b,b) \hspace{9cm}
\]
\begin{eqnarray*}
&=& d_{i_1}(a,d_{i_2}(a,\cdots d_{i_r}(a,d_1(a,\cdots d_1(a,d_1(a,b,b),d_n(a,b,b))\cdots ))\cdots )) \\
&=& d_{i_1}(a,d_{i_2}(a,\cdots d_{i_r}(a,a,t_{(n,\ldots,n)}(b,\ldots,b))\cdots )) \\
&=& d_{i_1}(a,d_{i_2}(a,\cdots d_{i_r - 1}(a,a,t_{(n,\ldots,n)}(b,\ldots,b))\cdots )) \\
&=& d_{i_1}(a,d_{i_2}(a,\cdots d_{i_r-1}(a,d_1(a,\cdots d_1(a,d_1(a,b,b),d_n(a,b,b))\cdots ))\cdots )) \\
&=& t_{(i_1,\ldots,i_{r}-1,1,\ldots,1,n,\ldots,n)}(b,\ldots,b,b).
\end{eqnarray*}
(4) Since $d_1$ is the first projection, this just reflects a substitution of the form $d_1(a,x,y)=d_1(a,w,z)$ in the nested definition of the polynomials.
\end{proof}
For the final step, we shall establish by descent on $0 \leq r \leq m-2$ that 
\begin{equation}\label{eq:induct4}
t_{(i_1,\ldots,i_r,n,\ldots,n)}(b,\ldots,b) \ [\theta,\ldots,\theta]_m \ t_{(i_1,\ldots,i_r,1,\ldots,1,n,\ldots,n)}(b,\ldots,b) 
\end{equation}
for all $(i_1,\ldots,i_r,1,\ldots,1,n,\ldots,n) \in S_m$. The theorem will then be concluded since at $r=0$ we have 
\begin{eqnarray*}
a = d_1(a,t_{(1,\ldots1,n)}(b,\ldots,b),t_{(n,\ldots,n)}(b,\ldots,b)) &=& t_{(1,\ldots,1,n,\ldots,n)}(b,\ldots,b) \\
&[\theta,\ldots,\theta]_m& t_{(n,\ldots,n,n,\ldots,n)}(b,\ldots,b).
\end{eqnarray*}
For the base case, take any $(i_1,\ldots,i_{m-2},n,n,\ldots,n) \in S_{m}$ and using Claim \ref{claim:reduce} and (\ref{eq:induct3}) calculate 
\begin{eqnarray*}
t_{(i_1,\ldots,i_{m-2},n,n,\ldots,n)}(b,\ldots,b,b) &[\theta,\ldots,\theta]_m& t_{(i_1,\ldots,i_{m-2},n,n,\ldots,n)}(b,\ldots,b,a) \\
&=& t_{(i_1,\ldots,i_{m-2},n-1,n,\ldots,n)}(b,\ldots,b,a) \\
&[\theta,\ldots,\theta]_m& t_{(i_1,\ldots,i_{m-2},n-1,n,\ldots,n)}(b,\ldots,b,b) \\
&=& t_{(i_1,\ldots,i_{m-2},n-2,n,\ldots,n)}(b,\ldots,b,b) \\
&\vdots& \\
&=& t_{(i_1,\ldots,i_{m-2},2,n,\ldots,n)}(b,\ldots,b,a) \\
&[\theta,\ldots,\theta]_m& t_{(i_1,\ldots,i_{m-2},2,n,\ldots,n)}(b,\ldots,b,b) \\
&=& t_{(i_1,\ldots,i_{m-2},1,n,\ldots,n)}(b,\ldots,b,b). 
\end{eqnarray*}
Now assume (\ref{eq:induct4}) holds for some $0 < r \leq m-2$. For $(i_1,\ldots,i_{r-1},1,\ldots,1,n,\ldots,n) \in S_m$ and calculate
\begin{eqnarray*}
t_{(i_1,\ldots,i_{r-1},n,n,\ldots,n)}(b,\ldots,b,b) &[\theta,\ldots,\theta]_m& t_{(i_1,\ldots,i_{r-1},n,1,\ldots,1,n,\ldots,n)}(b,\ldots,b,b) \\
&=& t_{(i_1,\ldots,i_{r-1},n-1,1,\ldots,1,n,\ldots,n)}(b,\ldots,b,b) \\
&[\theta,\ldots,\theta]_m& t_{(i_1,\ldots,i_{r-1},n-1,n,\ldots,n,n,\ldots,n)}(b,\ldots,b,b) \\
&=& t_{(i_1,\ldots,i_{r-1},n-2,n,\ldots,n,n,\ldots,n)}(b,\ldots,b,b) \\
&\vdots& \\
&=& t_{(i_1,\ldots,i_{r-1},2,1,\ldots,1,n,\ldots,n)}(b,\ldots,b,b) \\
&[\theta,\ldots,\theta]_m& t_{(i_1,\ldots,i_{r-1},2,n,\ldots,n,n,\ldots,n)}(b,\ldots,b,b) \\
&=& t_{(i_1,\ldots,i_{r-1},1,n,\ldots,n,n,\ldots,n)}(b,\ldots,b,b) \\
&=& t_{(i_1,\ldots,i_{r-1},1,1,\ldots,1,n,\ldots,n)}(b,\ldots,b,b).  
\end{eqnarray*}
This completes the inductive step and the theorem.
\end{proof}

It follows from (HC3) and the argument in \cite[Thm 6.3]{commod} that in a congruence modular variety, any n-difference term can be connected to the left-projection by a sequence of Gumm terms.

\begin{proposition}
If $\mathcal V$ is a congruence modular variety with a n-difference term $c(x,y,z)$, then there are ternary terms $d_1,\ldots,d_m$ such that $d_1,\ldots,d_m,c_n$ are Gumm terms for $\mathcal V$.
\end{proposition}

\begin{theorem}\label{thm:cdneutral}
If $\mathcal V$ is a congruence distributive variety, then $\mathcal V$ is n-commutator neutral for all $n \geq 2$.
\end{theorem}
\begin{proof}
From J\'{o}nsson \cite{jonsson}, $\mathcal V$ is congruence distributive if and only if there are terms $d_1,\ldots, d_n$ such that the following identities
\begin{eqnarray*}
x &=& d_1(x,y,z) \\
x &=& d_i(x,y,x) \ \ \ \ \ \ \text{ for all } i \\
d_i(x,x,z) &=& d_{i+1}(x,x,z) \ \ \ \text{ for even } i \\
d_i(x,z,z) &=& d_{i+1}(x,z,z) \ \ \ \text{ for odd } i \\
d_n(x,y,z) &=& z 
\end{eqnarray*}
hold throughout $\mathcal V$. If $\mathcal V$ is not neutral, then there is $\textbf{A} \in \mathcal V$ and $\theta_1,\ldots,\theta_m \in \Con \textbf{A}$ such that $[\theta_1,\ldots,\theta_m] < \theta_1 \wedge \cdots \wedge \theta_m$. If we set $\gamma = \theta_1 \wedge \cdots \wedge \theta_m$, then $[\gamma,\ldots,\gamma]_m < \gamma$. For any $(a,b) \in \gamma$, using the same definitions as in Theorem \ref{thm:cmdiff}, the argument establishes $a \, [\gamma,\ldots,\gamma]_m \, t_{(n,\ldots,n)}(b,\ldots,b)$ when $(n,\ldots,n) \in [n]^{k_m}$ independently of the term $q$. Since $d_n$ is the third projection throughout $\mathcal V$, we see that $t_{(n,\ldots,n)}(b,\ldots,b)=b$ which implies $\theta(a,b) \leq [\gamma,\ldots,\gamma]_m$. Since $(a,b) \in \gamma$ was arbitrary, $\gamma = [\gamma,\ldots,\gamma]_m$ which is a contradiction. 
\end{proof}

\section{The Higher Commutator in Varieties with Weak n-Difference Terms}\label{sec:weakdiffcomm}

In this section, we establish restricted versions of properties (HC5), (HC6) and (HC8) in varieties with weak n-difference terms. According to Theorem \ref{thm:ndifference}, these properties hold in varieties with a weak difference term provided the two-generated free algebras are finite, and for general varieties with a weak difference term if Conjecture \ref{conj:weakcomm} is resolved positively.

In the presence of weak n-difference terms in the variety $\mathcal V$, we can define several different functions $f: \Con \textbf{A} \rightarrow \Con \textbf{A}$, $\textbf{A} \in \mathcal V$, with uniform definitions throughout the variety, and terms $p_f(x,y,z)$ such that for every $\theta \in \Con \textbf{A}$, the term $p_f$ satisfies the Mal'cev identities on the blocks of the congruence $\theta/f(\theta)$ in the quotient algebra $\textbf{A}/f(\theta)$. It is reasonable to presume that by restricting to congruences in the interval $(f(\theta),\theta)$, some of the properties (HC4-HC8) which hold in Mal'cev varieties will also hold in this instance; indeed, this is the case.

The key ingredient in \cite{supernil} for proving properties of the higher commutator in Mal'cev varieties is the difference operator and its relation to absorbing polynomials in the algebra. In order to adapt the approach taken in \cite{supernil}, we begin be modifying those definitions suitable to the more general setting.

\begin{definition}
Let $\textbf{A}$ be an algebra and fix a ternary polynomial $m \in \Pol_3 \textbf{A}$. For $0 \in A$ and vectors $\textbf{a}_1 \in A^{n_1}, \ldots, \textbf{a}_k \in A^{n_k}$ with $n=n_1 + \cdots n_k$, recursively define the \textit{difference operators} as mappings $D^k_{o,(\textbf{a}_1,\ldots,\textbf{a}_k)}: \Pol_n \textbf{A} \rightarrow \Pol_n \textbf{A}$ in the following manner: for all $\textbf{x}_1 \in A^{n_1},\ldots, \textbf{x}_k \in A^{n_k}$,
\[
D^1_{o,\textbf{a}_1}(f)(\textbf{x}_1) := m(f(\textbf{x}_1),f(\textbf{a}_1),o)
\]
where $f(\textbf{y}_1) \in \Pol_{n_1} A$ and 
\[
D^{k+1}_{o,(\textbf{a}_1,\ldots,\textbf{a}_{k+1})}(f)(\textbf{x}_1,\ldots,\textbf{x}_{k+1}) := \hspace{6.5cm}
\]
\[
\begin{bmatrix}
D^{k}_{o,(\textbf{a}_1,\ldots,\textbf{a}_k)}(f(\textbf{y}_1,\ldots,\textbf{y}_{k+1})_{S_{k+1}}[\textbf{x}_{k+1}])(\textbf{x}_1,\ldots,\textbf{x}_k) \\
D^{k}_{o,(\textbf{a}_1,\ldots,\textbf{a}_k)}(f(\textbf{y}_1,\ldots,\textbf{y}_{k+1})_{S_{k+1}}[\textbf{a}_{k+1}])(\textbf{x}_1,\ldots,\textbf{x}_k) \\
o
\end{bmatrix}
\]
where $f(\textbf{y}_1,\ldots,\textbf{y}_{k+1}) \in \Pol_{k+1} \textbf{A}$ and $S_{k+1}$ is the set of coordinates corresponding to $\textbf{y}_{k+1}$.
\end{definition}

In \cite{supernil}, $m$ was taken to be a Mal'cev term for $\textbf{A}$. We will be interested in the case when $m$ is taken to be a weak n-difference term and even something a little more general (see Definition \ref{def:fterm}).

\begin{example}
Let $m,p \in \Pol_3 \textbf{A}$. For fixed $o,u,x_3 \in A$, 
\[
(y_1,y_2) \longmapsto m(p(y_1,y_2,x_3),p(y_1,y_2,u),0) \in \Pol_2 \textbf{A}.
\]
Then for the vector $(a_1,a_2) \in A^2$ the difference operator is given for all $x_1,x_2 \in A$ by
\[
D^{2}_{o,(a_1,a_2)}(m(p(y_1,y_2,x_3),p(y_1,y_2,u),o))(x_1,x_2) \hspace{4.4cm}
\]
\begin{eqnarray*}
&=& 
m \begin{bmatrix}
D^1_{o,a_1}(m(p(y_1,x_2,x_3),p(x_1,x_2,u),o))(x_1) \\
D^1_{o,a_1}(m(p(y_1,a_2,x_3),p(x_1,a_2,u),o))(x_1) \\
o
\end{bmatrix} \\
&=& m
\begin{bmatrix}
m \begin{bmatrix}
m(p(x_1,x_2,x_3),p(x_1,x_2,u),o)) \\
m(p(a_1,x_2,x_3),p(a_1,x_2,u),o)) \\
o
\end{bmatrix}  \\
m \begin{bmatrix}
m(p(x_1,a_2,x_3),p(x_1,a_2,u),o)) \\
m(p(a_1,a_2,x_3),p(a_1,a_2,u),o)) \\
o
\end{bmatrix} \\
o
\end{bmatrix} 
\end{eqnarray*}
\end{example}

The following definition slightly generalizes the notion of absorbing polynomials to include non-unary vectors and relativizes to non-trivial congruences.

\begin{definition}
Let $f(\textbf{x}_1,\ldots,\textbf{x}_m) \in \Pol_n \textbf{A}$ with $n=|S_1| + \cdots + |S_m|$ where $S_i$ denotes the subset of coordinates corresponding to the vector of variables $\textbf{x}_i$. Let $\theta \in \Con \textbf{A}$, $a \in A$ and $(\textbf{a}_1,\ldots,\textbf{a}_m) \in A^n$ such that the restriction $(\textbf{a}_1,\ldots,\textbf{a}_m)\restriction_{S_i} = \textbf{a}_i$ for $i=1,\ldots,m$. We say $f$ $\theta$-absorbs $(\textbf{a}_1,\ldots,\textbf{a}_n)$ to $a$ if $f(\textbf{b}_1,\ldots,\textbf{b}_m) \equiv_{\theta} a$ whenever $(\textbf{b}_1,\ldots,\textbf{b}_m)\restriction_{S_i} = \textbf{a}_i$ for some $1 \leq i \leq m$.
\end{definition}

If we only allow unary vectors, then this is just the notion of absorption in the quotient $\textbf{A}/\theta$ from Definition \ref{def:absorb}.

\begin{definition}\label{def:fterm}
Let $\textbf{A}$ be an algebra and $f:\Con \textbf{A} \rightarrow \Con \textbf{A}$ an order-preserving map. A term $p(x,y,z)$ is a weak $f$-term of $\textbf{A}$ if for all $\theta \in \Con \textbf{A}$ and $(a,b) \in \theta$, 
\[
p(a,b,b) \ f(\theta) \ a \ f(\theta) \ p(b,b,a).
\]
We say $p(x,y,z)$ is a weak $f$-term for a variety $\mathcal V$ if for each algebra in the variety there is an interpretation of $f$ in the congruence lattice such that $p$ is a weak $f$-term. 
\end{definition}

Any ternary term is a weak $f$-term where $f$ is the constant map which maps every congruence to the total relation. When $f$ is the constant map which maps every congruence to the identity relation, then a weak $f$-term is a Mal'cev term. Any ternary idempotent term is a weak $f$-term where $f$ is the identity map. The next example is closer to the intended application.

\begin{example}\label{ex:nilterms}
A weak m-difference term $q(x,y,z)$ for a variety $\mathcal V$ is a weak $f$-term where $f(\theta)=[\theta,\ldots,\theta]_m$. Let $\textbf{A} \in \mathcal V$. There is a clear extension of the derived series to the m-commutator: for $\theta \in \Con \textbf{A}$, define $_{m}[\theta]^{0}=\theta$ and $_{m}[\theta]^{k+1}=[_{m}[\theta]^{k},\ldots,_{m}[\theta]^{k}]_m$ for $k \in \mathds{N}$. Recursively define terms by setting 
\begin{equation}\label{eq:nilterms}
q_1(x,y,z):=q(x,y,z) \ \ , \ \ q_{n+1}(x,y,z):=q(x,q_n(x,y,y),q_n(x,y,z)) 
\end{equation}
for $n \geq 2$. A straight-forward argument shows that for any $(x,y) \in \theta$,  
\[
q_n(x,y,y) \ _{m}[\theta]^n \ x \ _{m}[\theta]^n \ q_n(y,y,x); 
\]
therefore, $q_n(x,y,z)$ is a weak $f$-term where $f(\theta) =$   $_{m}[\theta]^n$. Using (HC2) and (HC3), this implies that for any solvable or nilpotent algebra in the variety, some $q_n$ will be a Mal'cev term. 
\end{example}

\begin{definition}
Let $\textbf{A}$ be an algebra, $\theta,\theta_1,\ldots,\theta_n \in \Con \textbf{A}$ and $f:\Con \textbf{A} \rightarrow \Con \textbf{A}$ an order-preserving map. Define 
\begin{eqnarray*}
T^{\theta}_{n,f}(\theta_1,\ldots,\theta_n) &=& \{\left(g(\textbf{b}_1,\ldots,\textbf{b}_n),g(\textbf{a}_1,\ldots,\textbf{a}_n)\right) : \mathrm{ each } \ \textbf{a}_i \equiv_{\theta_i} \textbf{b}_i, g \in \Pol_n \textbf{A}, \\ 
&& g \ \mathrm{ restricted } \ \mathrm{ to } \ \prod_{i=1}^n \{\textbf{a}_i,\textbf{b}_i\} \ f(\theta) \mathrm{-absorbs } \ (\textbf{a}_1,\ldots,\textbf{a}_n) \}
\end{eqnarray*}
\end{definition}

In the special instance when $f(\theta)=[\theta,\ldots,\theta]_k$ and each $\theta_i=\theta$ we will write $T^{\theta}_{n,k}=T^{\theta}_{n,f}(\theta,\ldots,\theta)$. The following lemma is the main tool for establishing the theorems in this section.

\begin{lemma}\label{lem:central}
Let $\mathcal V$ be a variety with a weak $f$-term, $\textbf{A} \in \mathcal V$ and $\gamma, \theta,\theta_1,\ldots,\theta_n \in \Con \textbf{A}$ such that each $\theta_i \leq \theta$. If $\textbf{a}_1,\ldots, \textbf{a}_{n-1}, \textbf{b}_1,\ldots, \textbf{b}_{n-1} , \textbf{u}, \textbf{v}$ are vectors in $A$ and $p(\textbf{x}_1,\ldots,\textbf{x}_{n}) \in \Pol_n \textbf{A}$ such that 
\begin{enumerate}
 
	\item $\textbf{a}_i \equiv_{\theta_i} \textbf{b}_i$ and $\textbf{u} \equiv_{\theta_n} \textbf{v}$ for $i=1,\ldots,n-1$,

	\item $p(\textbf{z}_1,\ldots, \textbf{z}_{n-1},\textbf{u}) \ \equiv_{\gamma} \ p(\textbf{z}_1,\ldots, \textbf{z}_{n-1},\textbf{v})$ for all $(\textbf{z}_1,\ldots,\textbf{z}_{n-1}) \in \left\{\textbf{a}_1, \textbf{b}_1\right\} \times \cdots \times \left\{\textbf{a}_{n-1},\textbf{b}_{n-1}\right\} \backslash \left\{(\textbf{b}_1,\ldots,\textbf{b}_{n-1})\right\}$,
	
	\item $T^{\theta}_{n,f}(\theta_1,\ldots,\theta_n) \subseteq \gamma$, 

\end{enumerate}
then  
\[
p(\textbf{b}_1,\ldots,\textbf{b}_{n-1},\textbf{u}) \ f(\theta) \, \circ \, \gamma \, \circ \, f(\theta) \ p(\textbf{b}_1,\ldots,\textbf{b}_{n-1},\textbf{v}).
\]
\end{lemma}
\begin{proof}
Let $m(x,y,z)$ be the weak $f$-term for the variety $\mathcal V$. Note $p(\textbf{c}_1,\ldots,\textbf{c}_{n}) \equiv_{\theta} p(\textbf{b}_1,\ldots,\textbf{b}_{n-1},\textbf{u})$ for all $(\textbf{c}_1,\ldots,\textbf{c}_{n}) \in \left\{\textbf{a}_1, \textbf{b}_1\right\} \times \cdots \times \left\{\textbf{a}_{n-1},\textbf{b}_{n-1}\right\} \times \{\textbf{u},\textbf{v}\}$ since each $\theta_i \leq \theta$. For any choice of $w \equiv_{\theta} p(\textbf{b}_1,\ldots,\textbf{b}_{n-1},\textbf{u})$ define 
\[
t(\textbf{x}_1,\ldots,\textbf{x}_n) := \hspace{9cm}
\]
\[
D^{n-1}_{w,(\textbf{a}_1,\ldots,\textbf{a}_{n-1})}(m(p(\textbf{y}_1,\ldots,\textbf{y}_{n-1},\textbf{x}_n),p(\textbf{y}_1,\ldots,\textbf{y}_{n-1},\textbf{u}),w))(\textbf{x}_1,\ldots,\textbf{x}_{n-1}).
\]
We first show inductively on $n \geq 2$ that the data (1) and (2) implies the polynomial $t(\textbf{x}_1,\ldots,\textbf{x}_n)$ $f(\theta)$-absorbs $(\textbf{a}_1,\ldots,\textbf{a}_{n-1},\textbf{u})$ to $w$ when restricted to the set $\{\textbf{a}_1, \textbf{b}_1\} \times \cdots \times \{\textbf{a}_{n-1},\textbf{b}_{n-1}\} \times \{\textbf{u},\textbf{v}\}$. Consider vectors $(\textbf{c}_1,\ldots,\textbf{c}_{n-1},\textbf{c}_n) \in \{\textbf{a}_1, \textbf{b}_1\} \times \cdots \times \{\textbf{a}_{n-1},\textbf{b}_{n-1}\} \times \{\textbf{u},\textbf{v}\}$ such that some $\textbf{c}_i = \textbf{a}_i$ or $\textbf{c}_n=\textbf{u}$.

In the base binary case, we have
\begin{eqnarray*}
t(\textbf{c}_1,\textbf{u}) &=& 
m \begin{bmatrix}
m(p(\textbf{c}_1,\textbf{u}),p(\textbf{c}_1,\textbf{u}),w) \\
m(p(\textbf{a}_1,\textbf{u}),p(\textbf{a}_1,\textbf{u}),w) \\
w
\end{bmatrix} \ f(\theta) \ m(w,w,w)=w
\end{eqnarray*}
and
\begin{eqnarray*}
t(\textbf{a}_1,\textbf{c}_2) &=& 
m \begin{bmatrix}
m(p(\textbf{a}_1,\textbf{c}_2),p(\textbf{a}_1,\textbf{u}),w) \\
m(p(\textbf{a}_1,\textbf{c}_2),p(\textbf{a}_1,\textbf{u}),w) \\
w
\end{bmatrix} \ f(\theta) \ w
\end{eqnarray*}
since $m(p(\textbf{a}_1,\textbf{c}_2),p(\textbf{a}_1,\textbf{u}),w) \equiv_{\theta} m(w,w,w)=w$ and $m$ is a weak $f$-term. Now assume the result for $n-1 \geq 2$. Define the polynomials 
\[
q_1(\textbf{x}_1,\ldots,\textbf{x}_{n-1}):=p(\textbf{x}_1,\ldots,\textbf{x}_{n-2},\textbf{c}_{n-1},\textbf{x}_{n-1})
\] 
and 
\[
q_2(\textbf{x},\ldots,\textbf{x}_{n-1}):=p(\textbf{x}_1,\ldots,\textbf{x}_{n-2},\textbf{a}_{n-1},\textbf{x}_{n-1})
\] 
and note (2) holds for $q_1$ and $q_2$. Suppose $\textbf{c}_{i}=\textbf{a}_{i}$ for some $i \neq n-1$. Set $\bar{\textbf{y}}=(\textbf{y}_1,\ldots,\textbf{y}_{n-2})$. Then the inductive hypothesis applied to $q_1,q_2$ yields
\[
t(\textbf{c}_1,\ldots,\textbf{c}_n) \hspace{12cm}
\]
\begin{eqnarray*}
 &=& D^{n-1}_{w,(\textbf{a}_1,\ldots,\textbf{a}_{n-1})}(m(p(\bar{\textbf{y}},\textbf{y}_{n-1},\textbf{c}_n),p(\bar{\textbf{y}},\textbf{y}_{n-1},\textbf{u}),w))(\textbf{c}_1,\ldots,\textbf{c}_{n-1}) \\
&=& m
\begin{bmatrix}
D^{n-2}_{w,(\textbf{a}_1,\ldots,\textbf{a}_{n-2})}(m(q_1(\bar{\textbf{y}},\textbf{c}_n),q_1(\bar{\textbf{y}},\textbf{u}),w))(\textbf{c}_1,\ldots,\textbf{c}_{n-2}) \\
D^{n-2}_{w,(\textbf{a}_1,\ldots,\textbf{a}_{n-2})}(m(q_2(\bar{\textbf{y}},\textbf{c}_n),q_2(\bar{\textbf{y}},\textbf{u}),w))(\textbf{c}_1,\ldots,\textbf{c}_{n-2}) \\
w
\end{bmatrix} \\
&f(\theta)& m(w,w,w)=w
\end{eqnarray*}
In the case $\textbf{c}_{n-1}=\textbf{a}_{n-1}$, we see that $q_1=q_2$ and the calculation follows by the inductive hypothesis applied to $q_2$ and the fact that $m$ is a weak $f$-term. This completes the inductive step.

From $f(\theta)$-absorption, the definition of the relation $T^{\theta}_{n,f}(\theta_1,\ldots,\theta_n)$ and (3) we have 
\[
(t(\textbf{b}_1,\ldots,\textbf{v}), t(\textbf{a}_1,\ldots,\textbf{u})) \in T^{\theta}_{n,f}(\theta_1,\ldots,\theta_n) \subseteq \gamma
\]
which implies 
\[
t(\textbf{b}_1,\ldots,\textbf{v}) \ \gamma \ t(\textbf{a}_1,\ldots,\textbf{u}) \ f(\theta) \ p(\textbf{b}_1,\ldots,\textbf{b}_{n-1},\textbf{u}).
\]
To complete the argument we will show 
\[
t(\textbf{b}_1,\ldots,\textbf{v}) \ \gamma \circ f(\theta) \ p(\textbf{b}_1,\ldots,\textbf{b}_{n-1},\textbf{v})
\] 
by induction on $n$. Set $w=p(\textbf{b}_1,\ldots,\textbf{b}_{n-1},\textbf{u})$. In the binary case, 
\begin{eqnarray*}
t(\textbf{b}_1,\textbf{v}) &=& 
m \begin{bmatrix}
m(p(\textbf{b}_1,\textbf{v}),p(\textbf{b}_1,\textbf{u}),p(\textbf{b}_1,\textbf{u})) \\
m(p(\textbf{a}_1,\textbf{v}),p(\textbf{a}_1,\textbf{u}),p(\textbf{b}_1,\textbf{u})) \\
p(\textbf{b}_1,\textbf{u})
\end{bmatrix} \\ 
&\gamma&  
m \begin{bmatrix}
m(p(\textbf{b}_1,\textbf{v}),p(\textbf{b}_1,\textbf{u}),p(\textbf{b}_1,\textbf{u})) \\
m(p(\textbf{a}_1,\textbf{u}),p(\textbf{a}_1,\textbf{u}),p(\textbf{b}_1,\textbf{u})) \\
p(\textbf{b}_1,\textbf{u})
\end{bmatrix} \\
&f(\theta)& m(p(\textbf{b}_1,\textbf{v}),p(\textbf{b}_1,\textbf{u}),p(\textbf{b}_1,\textbf{u})) \\ 
&f(\theta)& p(\textbf{b}_1,\textbf{v}).
\end{eqnarray*}
Now assume the result for $n-1 \geq 2$. Since the difference operator is a recursive composition of polynomials, by setting $\bar{\textbf{y}}=(\textbf{y}_1,\ldots,\textbf{y}_{n-2})$ we see that   
\[
D^{n-2}_{w,(\textbf{a}_1,\ldots,\textbf{a}_{n-2})}(m(p(\textbf{y}_1,\ldots,\textbf{a}_{n-1},\textbf{v}),p(\textbf{y}_1,\ldots,\textbf{a}_{n-1},\textbf{u}),w))(\textbf{b}_1,\ldots,\textbf{b}_{n-2}) \hspace{3.2cm}
\]
\begin{eqnarray*}
&\gamma& D^{n-2}_{w,(\textbf{a}_1,\ldots,\textbf{a}_{n-2})}(m(p(\bar{\textbf{y}},\textbf{a}_{n-1},\textbf{u}),p(\bar{\textbf{y}},\textbf{a}_{n-1},\textbf{u}),w))(\textbf{b}_1,\ldots,\textbf{b}_{n-2}) \\
&=& D^{n-2}_{w,(\textbf{a}_1,\ldots,\textbf{a}_{n-2})}(m(q_2(\bar{\textbf{y}},\textbf{u}),q_2(\bar{\textbf{y}},\textbf{u}),w))(\textbf{b}_1,\ldots,\textbf{b}_{n-2}) \\
&f(\theta)& w 
\end{eqnarray*}
using the data in (2) and $f(\theta)$-absorption applied to $q_2$. If we define $q_1(\textbf{x}_1,\ldots,\textbf{x}_{n-1}):=p(\textbf{x}_1,\ldots,\textbf{x}_{n-2},\textbf{b}_{n-1},\textbf{x}_{n-1})$, then 
\[
t(\textbf{b}_1,\ldots,\textbf{b}_{n-1},\textbf{v})= \hspace{12cm}
\]
\begin{eqnarray*}
 &=& D^{n-1}_{w,(\textbf{a}_1,\ldots,\textbf{a}_{n-1})}(m(p(\bar{\textbf{y}},\textbf{y}_{n-1},\textbf{v}),p(\bar{\textbf{y}},\textbf{y}_{n-1},\textbf{u}),w))(\textbf{b}_1,\ldots,\textbf{b}_{n-1}) \\
&=& m
\begin{bmatrix}
D^{n-2}_{w,(\textbf{a}_1,\ldots,\textbf{a}_{n-2})}(m(q_1(\bar{\textbf{y}},\textbf{v}),q_1(\bar{\textbf{y}},\textbf{u}),w))(\textbf{b}_1,\ldots,\textbf{b}_{n-2}) \\
D^{n-2}_{w,(\textbf{a}_1,\ldots,\textbf{a}_{n-2})}(m(p(\bar{\textbf{y}},\textbf{a}_{n-1},\textbf{v}),p(\bar{\textbf{y}},\textbf{a}_{n-1},\textbf{u}),w))(\textbf{b}_1,\ldots,\textbf{b}_{n-2})\\
w
\end{bmatrix} \\
&\gamma \circ f(\theta)& m 
\begin{bmatrix}
q_1(\textbf{b}_1,\ldots,\textbf{b}_{n-2},\textbf{v}) \\
w \\
w
\end{bmatrix} \\
&=& m
\begin{bmatrix}
p(\textbf{b}_1,\ldots,\textbf{b}_{n-1},\textbf{v}) \\
w \\
w
\end{bmatrix} \ f(\theta) \ p(\textbf{b}_1,\ldots,\textbf{b}_{n-1},\textbf{v})\\
\end{eqnarray*}
using the inductive hypothesis on the $q_1$ part. This completes the induction and the proof of the lemma.
\end{proof}

\begin{proposition}\label{cor:basis}
Let $\mathcal V$ be a variety with a weak $f$-term and $\textbf{A} \in \mathcal V$. If $\theta,\theta_1,\ldots,\theta_n \in \Con \textbf{A}$ such that each $\theta_i \leq \theta$ and $f(\theta)=0$, then $\Cg^{\textbf{A}}(T^{\theta}_{n,f}(\theta_1,\ldots,\theta_n)) = [\theta_1,\ldots,\theta_n]$. 
\end{proposition}
\begin{proof}
If $(g(\textbf{a}_1,\ldots,\textbf{a}_n),g(\textbf{b}_1,\ldots,\textbf{b}_n)) \in T^{\theta}_{n,f}(\theta_1,\ldots,\theta_n))$, then 
\[
g(\textbf{z}_1,\ldots,\textbf{z}_{n-1},\textbf{a}_n) \ f(\theta) \ g(\textbf{z}_1,\ldots,\textbf{z}_{n-1},\textbf{b}_n)
\] 
for all $(\textbf{z}_1,\ldots,\textbf{z}_{n-1}) \in \left\{\textbf{a}_1, \textbf{b}_1\right\} \times \cdots \times \left\{\textbf{a}_{n-1},\textbf{b}_{n-1}\right\} \backslash \left\{(\textbf{b}_1,\ldots,\textbf{b}_{n-1})\right\}$ since some $\textbf{\textbf{z}}_i = \textbf{\textbf{a}}_i$ and $g(\textbf{x}_1,\ldots,\textbf{x}_n)$ $f(\theta)$-absorbs $(\textbf{a}_1,\ldots,\textbf{a}_n)$. Then 
\[
g(\textbf{a}_1,\ldots,\textbf{a}_{n-1},\textbf{a}_n) \ f(\theta) \ g(\textbf{b}_1,\ldots,\textbf{b}_{n-1},\textbf{a}_n) \ [\theta,\ldots,\theta]_n \ g(\textbf{b}_1,\ldots,\textbf{b}_{n-1},\textbf{b}_n)
\]
where the second congruence equivalence comes from the centralizer relation applied to the commutator congruence; therefore, $\Cg^{\textbf{A}}(T^{\theta}_{n,f}(\theta_1,\ldots,\theta_n)) \leq [\theta,\ldots,\theta]_{n}$ because $f(\theta)=0$.

For the reverse inequality, an application of Lemma \ref{lem:central} yields 
\[\textbf{A} \vDash \mathrm{C}(\theta,\ldots,\theta;\Cg^{\textbf{A}}(T^{\theta}_{n,f}(\theta_1,\ldots,\theta_n)))
\] 
which implies $[\theta,\ldots,\theta]_{n} \leq \Cg^{\textbf{A}}(T^{\theta}_{n,f}(\theta_1,\ldots,\theta_n))$.
\end{proof}

\begin{corollary}
Let $\mathcal V$ be a variety with a weak m-difference for some $m \geq 2$. If $\textbf{A} \in \mathcal V$ and $\theta \in \Con \textbf{A}$ is m-supernilpotent, then $\Cg^{\textbf{A}}(T^{\theta}_{n,m}) = [\theta,\ldots,\theta]_{n}$.
\end{corollary}

\begin{theorem}\label{thm:ftermprop}
Let $\mathcal V$ be a variety with a weak $f$-term and $\textbf{A} \in \mathcal V$. Let $\theta,\theta_1,\ldots,\theta_n \in \Con \textbf{A}$ such that $f(\theta) \leq [\theta_1,\ldots,\theta_n] \leq \theta_1 \vee \cdots \vee \theta_n \leq \theta$.
\begin{enumerate}

	\item For any $\gamma \in \Con \textbf{A}$, $\textbf{A} \vDash \mathrm{C}(\theta_1,\ldots,\theta_n; \gamma)$ if and only if 
	\[
	[\theta_1,\ldots,\theta_n] \leq \gamma.
	\]

	\item If $\eta \leq \theta_1 \wedge \cdots \wedge \theta_n$, then 
	\[
	[\theta_1/\eta,\ldots,\theta_n/\eta] = ([\theta_1,\ldots,\theta_n] \vee \eta )/\eta
	\]
	where the commutator on the left-side of the equality is computed in the quotient algebra $\textbf{A}/\eta$.

\end{enumerate}
\end{theorem}
\begin{proof}
(1) Necessity follows since the commutator is the intersection of all congruences $\gamma$ which satisfy $\textbf{A} \vDash \mathrm{C}(\theta_1,\ldots,\theta_n;\gamma)$. Suppose $[\theta_1,\ldots,\theta_n] \leq \gamma$ and assume $p \in \Pol \textbf{A}$ with vectors $\textbf{a}_i,\textbf{b}_i$ such that $\textbf{a}_i \ \theta_i \  \textbf{b}_i$ for $i \in [n]$ and 
\[
p(\textbf{z}_1,\ldots,\textbf{z}_{n-1},\textbf{a}_n) \ \gamma \  p(\textbf{z}_1,\ldots,\textbf{z}_{n-1},\textbf{b}_n)
\]
for all $(\textbf{z}_1,\ldots,\textbf{z}_{n-1}) \in \left\{\textbf{a}_1, \textbf{b}_1\right\} \times \cdots \times \left\{\textbf{a}_{n-1},\textbf{b}_{n-1}\right\} \backslash \left\{(\textbf{b}_1,\ldots,\textbf{b}_{n-1})\right\}$. By the first paragraph of Proposition \ref{cor:basis}, $T^{\theta}_{n,f}(\theta_1,\ldots,\theta_n) \subseteq [\theta_1,\ldots,\theta_n] \leq \gamma$ because $f(\theta) \leq [\theta_1,\ldots,\theta_n]$. Now Lemma \ref{lem:central} yields 
\[
p(\textbf{b}_1,\ldots,\textbf{b}_{n-1},\textbf{a}_n) \equiv_{\gamma} p(\textbf{b}_1,\ldots,\textbf{b}_{n-1},\textbf{b}_n)
\]
since $[\theta_1,\ldots,\theta_n] \circ \gamma \circ [\theta_1,\ldots,\theta_n] \subseteq \gamma$. This shows $\textbf{A} \vDash \mathrm{C}(\theta_1,\ldots,\theta_n;\gamma)$.

(2) This is a direct application of part (1) using Lemma \ref{lem:cen}(4) and follows the now standard argument \cite[Cor 6.3]{supernil}.
\end{proof}

In varieties with n-difference terms, we can always define weak $f$-terms where $f$ is given by the lower central series. The hypothesis $f(\theta) \leq [\theta_1,\ldots,\theta_n]$ can then be satisfied, for example, when each $\theta_i$ is a nested commutator built from a fixed congruence. Fix $\theta \in \Con \textbf{A}$ and define a set of evaluated higher commutators in the following manner: 
\[
\chi^{\theta}_{1}=\{\theta\} \ \ , \ \  \chi^{\theta}_{k+1} = \{[\alpha_1,\ldots,\alpha_n]: \text{ each } \alpha_i \in \bigcup_{i \leq k} \chi^{\theta}_{i}, n \geq 2\}
\]
for $k \in \mathds{N}$. Then $\Xi(\theta)= \bigcup \chi^{\theta}_{i}$ is the set of nested higher commutators recursively evaluated starting from $\theta$.

\begin{theorem}\label{thm:nweak}
Let $\mathcal V$ be a variety with weak m-difference terms for all $m \geq 2$ and $\textbf{A} \in \mathcal V$.
\begin{enumerate}
	
	\item Let $\theta_1,\ldots,\theta_n \in \Xi(\theta)$:
	
		\begin{enumerate}
	
			\item For any $\gamma \in \Con \textbf{A}$, $\textbf{A} \vDash \mathrm{C}(\theta_1,\ldots,\theta_n; \gamma)$ if and only if 
	\[
	[\theta_1,\ldots,\theta_n] \leq \gamma.
	\]
		
			\item If $\eta \leq \theta_1 \wedge \cdots \wedge \theta_n$, then 
	\[
	[\theta_1/\eta,\ldots,\theta_n/\eta] = ([\theta_1,\ldots,\theta_n] \vee \eta )/\eta
	\]
	where the commutator on the left-side of the equality is computed in the quotient algebra $\textbf{A}/\eta$.
	
			\item For $n > m \geq 1$, 
			\[
			[\theta_{1},\ldots,\theta_m,[\theta_{m+1},\ldots,\theta_{n}]] \leq [\theta_1,\ldots,\theta_n].
			\]
	
		\end{enumerate}
	
	\item For any $n \geq 2$, the class of n-supernilpotent algebras in $\mathcal V$ forms a subvariety. 
	
	\item If $\textbf{A}$ is n-supernilpotent, then $\textbf{A}$ is nilpotent of class n-1; consequently, each cover in a supernilpotent interval is an abelian interval.

\end{enumerate}
\end{theorem}
\begin{proof}
(1) Since each $\theta_i \in \chi^{\theta}_{k_i}$ is a nested composition of higher commutators, let $m_i$ denote the highest arity of the commutator which appears in the composition of $\theta_i$. Then $_{m_i}[\theta]^{k_i} \leq \theta_i$. If we let $M=\max\{m_i:i \in [n]\}$ and $K=\max\{k_i: i \in [n]\}$, then $[_{M}[\theta]^{K},\ldots,$ $_{M}[\theta]^{K}]_n \leq [\theta_1,\ldots,\theta_n]$. Let $r=\max\{n,M\}$ and $q(x,y,z)$ a weak r-difference term for $\mathcal V$. Then $q_{K+1}(x,y,z)$ is a weak $f$-term for $\mathcal V$ where $f(\theta)=[_{M}[\theta]^{K},\ldots,$ $_{M}[\theta]^{K}]_n$ and $q_{K+1}$ is defined by (\ref{eq:nilterms}) in Example \ref{ex:nilterms}. Now (1a) and (1b) follow by Theorem \ref{thm:ftermprop}.

(1c) Using the same $f$ and terms $q_{K+1}$  in the previous paragraph, we have the inequalities $f(\theta) \leq [\theta_1,\ldots,\theta_n]$ and $[f(\theta),\ldots,f(\theta)]_n \leq [f(\theta),\ldots,f(\theta)]_m \leq [\theta_1,\ldots,\theta_m]$. If we set $\eta = [f(\theta),\ldots,f(\theta)]_n$, then $q_{K+2}(x,y,z)$ is a weak $g$-difference term for $\mathcal V$ where $g(\theta)=\eta$. The inequalities show the hypothesis in Theorem \ref{thm:ftermprop} is satisfied for the congruences $\theta_1,\ldots,\theta_n, [\theta_{m+1},\ldots,\theta_{n}]$. Using this, we can show that it suffices to establish the inequality in (1c) in the quotient $\textbf{A}/\eta$. To see this, repeatedly using Theorem \ref{thm:ftermprop}(2) we have  
\begin{eqnarray*}
[\theta_1,\ldots,\theta_m,[\theta_{m+1},\ldots,\theta_n]]/\eta &=& \left([\theta_1,\ldots,\theta_m,[\theta_{m+1},\ldots,\theta_n]] \vee \eta\right)/\eta \\
&=& \left[\theta_1/\eta,\ldots,\theta_m/\eta,[\theta_{m+1},\ldots,\theta_n]/\eta\right] \\
&=& [\theta_1/\eta,\ldots,\theta_m/\eta,\left([\theta_{m+1},\ldots,\theta_n] \vee \eta \right)/\eta] \\
&=& \left[\theta_1/\eta,\ldots,\theta_m/\eta,[\theta_{m+1}/\eta,\ldots,\theta_n/\eta]\right] \\
&\leq& [\theta_1/\eta,\ldots,\theta_n/\eta] \\
&=& \left([\theta_1,\ldots,\theta_n] \vee \eta \right)/\eta = [\theta_1,\ldots,\theta_n]/\eta. 
\end{eqnarray*}
Then by the Correspondence Theorem, the inequality in (1c) follows. Note, in the quotient $\textbf{A}/\eta$ we have $g(\theta/\eta)=0$; therefore, reusing the same names, we may assume $\theta \in \Con \textbf{A}$ such that $g(\theta)=0$. By Proposition \ref{cor:basis}, we have generating sets 
\begin{eqnarray*}
T'&=&T^{\theta}_{m+1,0}(\theta_1,\ldots,\theta_m,[\theta_{m+1},\ldots,\theta_n]) \\
T''&=&T^{\theta}_{n-m,0}(\theta_{m+1},\ldots,\theta_n) \\
T'''&=&T^{\theta}_{n,0}(\theta_1,\ldots,\theta_n) 
\end{eqnarray*}
for $[\theta_1,\ldots,\theta_m,[\theta_{m+1},\ldots,\theta_n]]$, $[\theta_{m+1},\ldots,\theta_n]$ and $[\theta_1,\ldots,\theta_n]$, respectively. We show $T' \subseteq T'''$.

Let $(a,b)=(h(\textbf{a}_1,\ldots,\textbf{a}_m,\textbf{a}_{m+1}),h(\textbf{b}_1,\ldots,\textbf{b}_m,\textbf{b}_{m+1})) \in T'$ where $\textbf{a}_i \, \theta_i \, \textbf{b}_i$ for $i \in [m]$, $\textbf{a}_{m+1} \, [\theta_{m+1},\ldots,\theta_n] \, \textbf{b}_{m+1}$ and $h$ absorbs $(\textbf{a}_1,\ldots,\textbf{a}_m,\textbf{a}_{m+1})$. For simplicity in the presentation of the argument, we will assume the tuples $\textbf{a}_{m+1}=a_{m+1}$, $\textbf{b}_{m+1}=b_{m+1}$ are singletons. The general argument proceeds in the same manner but with a multiplicity of indices.

Since $(a_{m+1},b_{m+1}) \in [\theta_{m+1},\ldots\theta_n]$ and the generating set $T''$ is closed under unary polynomials, there exists a sequence $a_{m+1}=c_{1},c_2,c_3,\ldots,c_{k},c_{k+1}=b_{m+1}$, polynomials $u_{i}(\textbf{x}^{i}_{m+1},\ldots,\textbf{x}^{i}_{n})$ and multivectors $(\textbf{e}^{i}_{m+1},\ldots,\textbf{e}^{i}_{n}),(\textbf{d}^{1}_{m+1},\ldots,\textbf{d}^{i}_{n})$ for $i \in [k]$ such that at each position $i \in [k]$ in the sequence we have 
\[
\textbf{e}^{i}_{m+1} \ \theta_{m+1} \textbf{d}^{i}_{m+1}, \ldots, \textbf{e}^{i}_{n} \ \theta_n \ \textbf{d}^{i}_{n},
\]
\[
c_i= u_{i}(\textbf{e}^{i}_{m+1},\ldots,\textbf{e}^{i}_{n}),\ \ \ c_{i+1}=u_{i}(\textbf{d}^{i}_{m+1},\ldots,\textbf{d}^{i}_{n})
\]
and $(u_{i}(\textbf{e}^{i}_{m+1},\ldots,\textbf{e}^{i}_{n}), u_{i}(\textbf{d}^{i}_{m+1},\ldots,\textbf{d}^{i}_{n})) \in T''$.
Note $\{c_1,\ldots,c_{k+1}\}$ is contained in a single $[\theta_{m+1},\ldots,\theta_n]$-block, and so $q_{K+2}$ satisfies the Mal'cev identities restricted to the sequence.  We can always assume $u_{i}$ absorbs $(\textbf{e}^{i}_{m+1},\ldots,\textbf{e}^{i}_{n})$ to $c_i$. Suppose this is not the case and $u_{i}$ absorbs $(\textbf{d}^{i}_{m+1},\ldots,\textbf{d}^{i}_{n})$ to $c_{i+1}$. Define a polynomial $s(\textbf{x}^{i}_{m+1},\ldots,\textbf{x}^{i}_m):= q_{K+2}(c_2,u_i(\textbf{x}^{i}_{m+1},\ldots,\textbf{x}^{i}_{n}),c_1)$. Then $s$ absorbs $(\textbf{d}^{i}_{m+1},\ldots,\textbf{d}^{i}_{n})$ to $c_{i}$. Then we have the required pattern after relabeling. Observe that this does not change the length of the sequence.

Of all such sequences which witness the congruence generation of $(a_{m+1},b_{m+1})$, take one of minimal length in $k$. We claim $k=1$. If this is not the case and $k>1$, define a polynomial 
\[
t((\textbf{x}^{1}_{m+1},\textbf{x}^{2}_{m+1}),\ldots,(\textbf{x}^{1}_{n},\textbf{x}^{2}_{n})):= q_{K+2}(u_1(\textbf{x}^{1}_{m+1},\ldots,\textbf{x}^{1}_{n}),c_1,u_2(\textbf{x}^{2}_{m+1},\ldots,\textbf{x}^{2}_{n}))
\]
and observe $(\textbf{e}^{1}_{m+1},\textbf{e}^{2}_{m+1}) \ \theta_{m+1} \ (\textbf{d}^{1}_{m+1},\textbf{d}^{2}_{m+1}),\ldots, (\textbf{e}^{1}_{n},\textbf{e}^{2}_{n}) \ \theta_{n} \ (\textbf{d}^{1}_{n},\textbf{d}^{2}_{n})$. It is not hard to see that $t$ absorbs $((\textbf{e}^{1}_{m+1},\textbf{e}^{2}_{m+1}),\ldots,(\textbf{e}^{1}_{n},\textbf{e}^{2}_{n}))$ to $c_1$ and therefore, 
\[
(c_1,c_3) = \left(t\left((\textbf{e}^{1}_{m+1},\textbf{e}^{2}_{m+1}),\ldots,(\textbf{e}^{1}_{n},\textbf{e}^{2}_{n})\right),t\left((\textbf{d}^{1}_{m+1},\textbf{d}^{2}_{m+1}),\ldots,(\textbf{d}^{1}_{n},\textbf{d}^{2}_{n})\right)\right) \in T''.
\] 
Since we can use the polynomial $t$ to shorten the sequence, it must be that $k=1$.

Finally, we can then take $s(\textbf{x}_1,\ldots,\textbf{x}_n) := h(\textbf{x}_1,\ldots,\textbf{x}_{m},u_1(\textbf{x}_{m+1},\ldots,\textbf{x}_{n}))$ which absorbs $(\textbf{a}_1,\ldots,\textbf{a}_{m},\textbf{e}_{m+1},\ldots,\textbf{e}_{m})$ to $a$. Then 
\[
(a,b)=(s(\textbf{a}_1,\ldots,\textbf{a}_{m},\textbf{e}_{m+1},\ldots,\textbf{e}_{m}),s(\textbf{b}_1,\ldots,\textbf{b}_{m},\textbf{d}_{m+1},\ldots,\textbf{d}_{m})) \in T'''.
\]

(2) That n-supernilpotence is preserved by passing to direct products and subalgebras follows from the fact that after adding appropriate constants, the centralizer condition defining the higher commutators is a quasi-equation for each polynomial. Preservation by homomorphisms follows from part (1b) above.

(3) The first statement follows by repeated application of the inequality $ \left[\theta,\ldots,\theta\right]_n$ $\geq [\theta,[\theta,\ldots,\theta]_{n-1}]$ from (1c). The second statement can then be argued as in Lemma \ref{lem:abelinterv} using (1a) and (1c) together.
\end{proof}

\begin{remark}
If in varieties with a weak difference term, we can prove the restricted form of (HC8) which is in Theorem \ref{thm:nweak}(1c), then the terms constructed in (\ref{eq:nilterms}) will be weak n-difference terms. This is one approach to establishing Conjecture \ref{conj:weakcomm}.
\end{remark}

The following corollary implies that when considering commutators of the form $[\theta,\ldots,\theta]_n$ the centralizer relation is determined by its restriction to n-ary polynomials; in particular, the generating set in Proposition \ref{cor:basis} can be taken over unary vectors. Using Lemma \ref{lem:central}, the argument can proceed in exactly the same manner as in \cite[Lem 5.2-5.4]{supernil} and so we omit the rather long calculation since no new insight is offered.

\begin{corollary}
Let $\mathcal V$ be a variety with weak m-difference terms for all $m \geq 2$, $\textbf{A} \in \mathcal V$ and $\theta, \gamma \in \Con \textbf{A}$. Then $\textbf{A} \vDash \mathrm{C}(\theta,\ldots,\theta; \gamma)$ if and only if the centralizer relation holds using only unary vectors and n-ary polynomials in the preamble.
\end{corollary}

\begin{acknowledgements}
I would like to thank Andrew Moorhead and Jakub Opr\v{s}al for intriguing and enthusiastic discussions about the higher commutator during the Vanderbilt Workshop on Structure and Complexity in Universal Algebra held September 19 - 30, 2016 in Nashville, TN. The author was supported in part by National Natural Science Foundation of China Research Fund for International Young Scientists \#11650110429
\end{acknowledgements}

\end{document}